\documentclass[smallextended,envcountsect,final]{svjour3}
\smartqed

\usepackage{graphicx}
\usepackage{lscape}
\usepackage{amsmath}
\usepackage{latexsym}
\usepackage{fancybox}
\usepackage{floatrow}
\usepackage{paralist}
\usepackage{lineno}
\usepackage{enumerate}
\usepackage{cases}
\usepackage[normalem]{ulem}
\usepackage{color}
\usepackage[colorlinks,linkcolor=blue,anchorcolor=blue,citecolor=blue]{hyperref}
\usepackage{url}
\usepackage{longtable}
\usepackage{booktabs}

\newtheorem{assumption}{Assumption}
\newtheorem{algorithm}{Algorithm}

\DeclareMathOperator*{\argmin}{{\rm arg\,min}}
\newcommand{\diag}{{\rm Diag}}

\newcommand{\dom}{{\rm dom}}
\newcommand{\A}{{\mathcal A}}
\newcommand{\B}{{\mathcal B}}
\newcommand{\D}{{\mathcal D}}
\newcommand{\E}{{\mathcal E}}

\newcommand{\cH}{{\mathcal H}}
\newcommand{\I}{{\mathcal I}}
\newcommand{\K}{{\mathcal K}}
\newcommand{\cL}{{\mathcal L}}
\newcommand{\M}{{\mathcal M}}
\newcommand{\N}{{\mathcal N}}

\newcommand{\cP}{{\mathcal P}}
\newcommand{\Q}{{\mathcal Q}}

\newcommand{\cS}{{\mathcal S}}
\newcommand{\T}{{\mathcal T}}

\newcommand{\X}{{\mathcal X}}
\newcommand{\Y}{{\mathcal Y}}
\newcommand{\Z}{{\mathcal Z}}
\newcommand{\ds}{\displaystyle}

\newcommand{\mc}{\multicolumn}

\newcommand{\bx}{{\bf x}}
\newcommand{\by}{{\bf y}}
\newcommand{\bz}{{\bf z}}
\newcommand{\bs}{{\bf s}}
\newcommand{\bc}{{\bf c}}
\newcommand{\bb}{{\bf b}}
\newcommand{\bu}{{\bf u}}
\newcommand{\bX}{{\bf X}}
\newcommand{\bY}{{\bf Y}}

\definecolor{mycolor}{RGB}{195,220,195}
\usepackage{pagecolor}
\journalname{MPC}
\begin{document}
%All acknowledgements should be placed in the back of the paper before References.
\title{A Generalized Alternating Direction Method of Multipliers with Semi-Proximal Terms for Convex Composite Conic Programming}
\titlerunning{Generalized ADMM}

\author{Yunhai Xiao \and  Liang Chen \and Donghui Li}
%%%%%%%%%%%%%%%%%%%%%%%%%%%%%%%%%%%%%%%%%%%%%%%%%%%%%%%%%%%%%%%%%%%%%%%
\institute{
Yunhai Xiao,  Corresponding author
\at
Institute of Applied Mathematics, College of Mathematics and Statistics,
Henan University, Kaifeng 475000, China.\\
\email{yhxiao@henu.edu.cn}
\and
Liang Chen
\at
College of Mathematics and Econometrics, Hunan University, Changsha 410082, China.\\
\email{chl@hnu.edu.cn}
\and
Donghui Li
\at
School of Mathematical Sciences, South China
Normal University, Guangzhou 510631, China. \\
\email{dhli@scnu.edu.cn}}
%%%%%%%%%%%%%%%%%%%%%%%%%%%%%%%%%%%%%%%%%%%%%%%%%%%%%%%%%%%%%%%%%%%%%%%%
\date{Received: date / Accepted: date}

%%%%%%%%%%%%%%%%%%%%%%%%%%%%%%%%%%%%%%%%%%%%%%%%%%%%%%%%%%%%%%%%%%%%%%%%
\maketitle
\begin{abstract}
In this paper, we propose a generalized alternating direction method of multipliers (ADMM) with semi-proximal terms for solving a class of convex composite conic optimization problems, of which some are high-dimensional, to moderate accuracy. Our primary motivation is that this method, together with properly chosen semi-proximal terms, such as those generated by the recent advance of block symmetric Gauss-Seidel technique, is capable of tackling these problems. Moreover, the proposed method, which relaxes both the primal and the dual variables in a natural way with a common relaxation factor in the interval of $(0,2)$, has the potential of enhancing the performance of the classic ADMM. Extensive numerical experiments on various doubly non-negative semidefinite programming problems, with or without inequality constraints, are conducted. The corresponding results showed that all these multi-block problems can be successively solved, and the advantage of using the relaxation step is apparent.
\end{abstract}
%%%%%%%%%%%%%%%%%%%%%%%%%%%%%%%%%%%%%%%%%%%%%%%%%%%%%%%%%%%%%%%%%%%%%%%%%%

%%%%%%%%%%%%%%%%%%%%%%%%%%%%%%%%%%%%%%%%%%%%%%%%%%%55
\section{Introduction}\label{intr}
%%%%%%%%%%%%%%%%%%%%%%%%%%%%%%%%%%%%%%%%%%%%%%%%%%%%%%%%%%%%%%%%%%%%%%%%
Let $\bf X$ and $\bf Y$ be two finite-dimensional real Euclidean spaces each equipped with an inner product $\langle\cdot,\cdot\rangle$ and its induced norm $\|\cdot\|$.
The primary motivation of this paper is to develop an efficient first-order method to solve the following convex composite conic programming problem
\begin{eqnarray}\label{probc}
\begin{array}{ll}
\ds \min_{\bx} \ & \theta(\bx)+\langle \bc,\bx\rangle\\[1mm]
\mbox{s.t.} & \cH \bx-\bb\in\mathcal{C}, \ \bx\in \K,
\end{array}
\end{eqnarray}
where $\theta:\bX \to(-\infty,+\infty]$ is a closed proper convex function,
$\cH:\bX\to\bY$ is a linear map with its adjoint denoted by $\cH^*$,
$\bc\in\bX$ and $\bb\in\bY$ are given vectors, $\mathcal{C}\subseteq \bY$ and $\K\subseteq\bX$ are two closed convex cones.

By simple calculations, one can observe that the dual of problem (\ref{probc}) can be recast as
%\begin{equation}\label{dualprob}
%\begin{array}{ll}
%\ds\min_{\bs,\bz,\by} \ & \theta^*(-\bs)-\langle \bb,\by\rangle\\[1mm]
%\mbox{s.t.} & \bs+\bz+\cH^*\by=\bc,\\[1mm]
%            &\bz\in\K^*, \ \by\in\mathcal{C}^*,% \ \bs\in\X,
%\end{array}
%\end{equation}
\begin{equation}
\label{dual2}
\begin{array}{ll}
\ds\min_{\bs,\bz,\by} \ & \theta^*(-\bs)+\delta_{\K^*}(\bz)+\delta_{\mathcal{C}^*}(\by)-\langle \bb,\by\rangle\\[1mm]
\mbox{s.t.} & \bs+\bz+\cH^*\by=\bc,% \ s\in \X,
\end{array}
\end{equation}
where $\mathcal{C}^*$ and $\K^*$ are the dual cones of $\mathcal{C}$ and $\K$, respectively, $\theta^*$ is the Fenchel conjugate of $\theta$,
$\delta_{\mathcal{C}^*}(\cdot)$ and $\delta_{\K^*}(\cdot)$ are the indicator functions of $\mathcal{C}^*$ and $\K^*$, respectively. Let $\D:\bX\rightarrow\bX$ be a given nonsingular linear operator and $\D^*$ being its adjoint. Then, we can reformulate problem (\ref{dual2}) equivalently as
\begin{equation}\label{dual3}
\begin{array}{ll}
\ds\min_{\bs,\bz,\by,\bu} \ & \theta^*(-\bs)+\delta_{\K^*}(\bz)+\delta_{\mathcal{C}^*}(\bu)-\langle \bb,\by\rangle\\
\mbox{s.t.} & \bs+\bz+\cH^*\by=\bc,\\
            & \D^*(\bu-\by)=0.
\end{array}
\end{equation}
Obviously, problem \eqref{dual3} can be viewed (by setting $(\bs,\bu)$ as one block-variable) as an example of the following general multi-block convex composite problem
\begin{equation}
\label{gmodel}
\begin{array}{ll}
\min \ & f_1(y_1)+f_2(y_1,y_2,\ldots,y_p)+g_1(z_1)+g_2(z_1,z_2,\ldots,z_q)\\
\mbox{s.t.} &\A_1^*y_1+\A_2^*y_2+\ldots+\A_p^*y_p+\B_1^*z_1+\B_2^*z_2+\ldots+\B_q^*z_q=c,
\end{array}
\end{equation}
where $p$ and $q$ are two given nonnegative integers, $\X,\Y_1,\ldots,\Y_p,\Z_1,\ldots,\Z_q$ are finite-dimensional real Euclidean spaces each endowed with an inner product $\langle \cdot,\cdot\rangle$ and its induced norm $\|\cdot\|$, $c\in\X$ being the given data, $\A_i,i=1,\ldots,p,$ and $\B_j,j=1,\ldots,q,$ are linear maps from $\Y_i$ and $\Z_j$, respectively, to $\X$, $f_1:\Y_1\rightarrow(-\infty,+\infty]$ and $g_1:\Z_1\rightarrow(-\infty,+\infty]$ are simple closed proper convex functions, $f_2:\Y\rightarrow(-\infty,+\infty)$ and $g_2:\Z\rightarrow(-\infty,+\infty)$ are  convex quadratic functions, in which $\Y$ and $\Z$ are defined by $\Y:=\Y_1\times\Y_2\times\ldots\times\Y_p$ and $\Z:=\Z_1\times\Z_2\times\ldots\times\Z_q$.

For further simplicity, we define the linear maps $\A:\X\rightarrow\Y$ and $\B:\X\rightarrow\Z$ such that their adjoints  are given by
$$
\A^*y=\sum_{i=1}^p\A_i^*y_i \quad \forall y\in\Y, \quad\mbox{and}\quad \B^*z=\sum_{j=1}^q\B_j^*z_j \quad \forall z\in\Z.
$$
Then, the model (\ref{gmodel}) falls into the following more general framework
\begin{equation}\label{prob}
\min_{y\in\Y,z\in\Z}\Big\{f(y)+g(z)\quad\mbox{s.t.}\quad\A^* y+\B^* z=c\Big\},
\end{equation}
where $f:\Y \to(-\infty,+\infty]$ and $g:\Z\to(-\infty,+\infty]$
are two closed proper convex functions,
$\A:\X\rightarrow\Y$ and $\B:\X\rightarrow\Z$
are two linear maps with their adjoints $\A^*$ and $\B^*$, respectively,  and $c\in\X$ is a given vector.

Let $\sigma>0$ be the penalty parameter, the augmented Lagrangian function of problem \eqref{prob} is defined by, for any $(x,y,z)\in\X\times\Y\times\Z$,
\begin{equation}
\cL_{\sigma}(y,z;x):=f(y)+g(z)-\langle x,\A^*y+\B^{*}z-c\rangle+\frac{\sigma}{2}\|\A^*y+\B^*z-c\|^2.
\end{equation}
Let $\tau>0$ be the step-length and choose an initial point $(x^{0},y^{0},z^{0})\in\X\times(\dom\, f)\times(\dom\, g)$. The classic alternating direction method of multipliers (ADMM) scheme takes the following form, for $k=0,1,\ldots,$
\begin{equation}
\label{admm}
\left\{
\begin{array}{ll}
\displaystyle
y^{k+1}
\in \argmin_{y}
\cL_{\sigma}(y,z^{k};x^{k}),
\\[2mm]
\displaystyle
z^{k+1}
\in \argmin_{z}
\cL_{\sigma}(y^{k+1},z;x^{k}),
\\[2mm]
\displaystyle
x^{k+1}: =x^{k}-\tau\sigma\big(\A^{*}y^{k+1}+\B^{*}z^{k+1}-c\big).
\end{array}
\right.
\end{equation}
Generally, the step-length $\tau$ in \eqref{admm} can  be chosen in the interval of $(0,(1+\sqrt{5})/2)$, and the thumb of the choice for $\tau$ in
numerical computations is the golden ratio of $1.618$. Moreover, we should mention that the iteration scheme \eqref{admm} may not be well-defined, and one may refer to \cite{admmnote} for details.

The classic ADMM, which is closely related to the method of multipliers by Hestenes \cite{hestenes69}, Powell \cite{powell} and Rockafellar \cite{rockafellar} for solving constrained convex optimization problems, was originally proposed by Glowinski and Marroco \cite{GLOWINSKI75} and Gabay and Mercier \cite{GABAY76} in the mid-1970s.
Early proofs for the convergence of the classic ADMM under certain conditions were given by Gabay and Mercier \cite{GABAY76}, Glowinski \cite{GLOWINSKI80} and Fortin and Glowinski \cite{FORTIN83}. One may refer to \cite{globook14} for a note on the historical development of the ADMM.
In early-1980s, it was shown by Gabay \cite{GABAY83} that the classic ADMM with the unit step-length is a special case of the Douglas-Rachford splitting method applied to finding the roots of the sum of two maximal monotone operators.
In early-1990s, Eckstein and Bertsekas \cite{ECKSTEIN92} showed that the Douglas-Rachford splitting method itself is an instance of the proximal point algorithm \cite{roc76} applied to a specially generated splitting operator.
Based on this observation, and for the purpose of improving the performance of the classic ADMM with the unit step-length,
Eckstein and Bertsekas \cite{ECKSTEIN92} presented a generalized ADMM.
For a recent survey, one may refer to \cite{eck12}.

Along a different line, Eckstien \cite{ECKSTEINOMS} introduced proximal terms to the ADMM to ensure that each of the subproblems involved admits a unique solution, without further assumptions on the objective functions and the constraints.
Recently, a semi-proximal ADMM (sPADMM) with the step-length up to the golden ratio of $1.618$ was proposed by Fazel et al. \cite{SEMP13}.
This extension has attracted much attention in the last couple of years in solving multi-block linearly constrained linear and quadratic convex semidefinite programming problems \cite{CHENL,XDLIMP,lithesis,YANGSIAM,limajorize} to moderate accuracy.
The successful applications of the two-block ADMM with semi-proximal terms to solving various multi-block problems for moderate accuracy inevitably inspire us to consider
using generalized ADMM with semi-proximal terms to solve the convex composite conic programming (\ref{probc}).
Moreover, the advantage of using the semi-proximal terms instead of the positive definite proximal terms is that the resulting algorithm not only reduces to the original generalized ADMM when the proximal terms are absent but also is more convenient in solving convex composite conic programming if the block symmetric Gauss-Seidel (sGS) iteration techinique invented by Li, Sun and Toh \cite{lisgs} is adopted as in \cite{CHENL,XDLIMP,YANGSIAM}.

In this paper, we proposed a variant of generalized ADMM with semi-proximal terms, which can be used to solve problem \eqref{gmodel} if the block sGS technique is involved.
This variant is based on a crucial observation made by Chen \cite{CHENPHD} that the generalized ADMM of Eckstein and Bertsekas \cite{ECKSTEIN92} can be reformulated equivalently as an ADMM with an extra relaxation step on both the primal and the dual variables with the relaxation factor lying in the open interval of $(0,2)$.
Furthermore, the semi-proximal terms used in this paper is in a pretty natural way due to the fact that only the most recent values of variables are always in it.
To derive a more general theoretical convergence result, we particularly concentrate on the convergence analysis of the generalized ADMM with semi-proximal terms for solving the `two-block' problem (\ref{prob}).
Then, we illustrate how the proposed algorithm can be used to solve multi-block problems.
At last, we conduct numerical experiments on a class linear  doubly non-negative (DNN) semidefinite programming (SDP)
problems with or without inequality constraints. The corresponding numerical results illustrate that the proposed method, together with the sGS technique, can solve these problems not only effectively but also efficiently.

The remaining parts of this paper are organized as follows. In Section \ref{presult}, we provide some preliminary results, which will be used later, and give a quick review of some variants of the classic ADMM.
In Section \ref{spre2}, we present a generalized ADMM with semi-proximal terms for solving the $2$-block problem \eqref{prob}. Then, we discuss how to apply the proposed method, with properly chosen semi-proximal terms, to solving multi-block problems in Section \ref{appl}.
In Section \ref{glob}, we focus on the convergence analysis of the proposed generalized ADMM.
In Section \ref{numer}, we are devoted to the implementation issues and numerical experiments of using the proposed method for solving DNN-SDP.
The corresponding numerical results are provided.
We conclude this paper in Section \ref{result}.

%%%%%%%%%%%%%%%%%%%%%%%%%%%%%%%%%%%%%%%%%%%%%%%%%%%%%%%%%%%
\section{Preliminaries}\label{presult}
%%%%%%%%%%%%%%%%%%%%%%%%%%%%%%%%%%%%%%%%%%%%%%%%%%%%%%%%%%%%%%%%%%%%%%%
We adopt the notation that $w:=(x,y,z)\in\X\times\Y\times\Z$ for convenience.
Moreover, for any two finite-dimensional real Euclidean spaces $\X'$ and $\X''$ each endowed with an inner product $\langle\cdot,\cdot\rangle$, we take the convention that
$$
\langle (x'_1,x''_1),(x'_2,x''_2)\rangle:=
\langle x'_1,x'_2\rangle
+\langle x''_1,x''_2\rangle,
\quad
\forall x_1',x_2'\in\X', x_1'',x_2''\in\X'',
$$
and the corresponding norm in the product space induced by this inner product is also denoted by $\|\cdot\|$.
We use $\I$ to denote the identity linear operator, whose domain and range can be deduced from the context.

We denote the effective domains of $f$ and $g$ by $\dom\, f$ and $\dom\, g$, respectively, and use
$\partial f$ and $\partial g$ to denote their subdifferential mappings. It is well-known that $\partial f$ and $\partial g$ are maximal monotone operators \cite{rocbook} and there exist two self-adjoint positive semidefinite linear operators $\Sigma_f:\Y\to\Y$ and $\Sigma_g:\Z\to\Z$, such that for any
$y,y'\in\Y$ and $z,z'\in\Z$ with $u\in\partial f(y)$, $u'\in\partial f(y')$, $v\in\partial g(z)$, and $v'\in\partial g(z')$,
\begin{equation}
\label{ineqx}
\langle u-u',y-y'\rangle\geq \|y-y'\|_{\Sigma_f}^2
\quad
\mbox{and}
\quad
\langle v-v',z-z'\rangle\geq \|z-z'\|_{\Sigma_g}^2.
\end{equation}

The Lagrangian function of problem \eqref{prob} is defined by
$$
\cL(y,z;x):=f(y)+g(z)-\langle x,\A^*y+\B^{*}z-c\rangle, \ \forall (x,y,z)\in\X\times\Y\times\Z.
$$
A vector $( x, y, z)\in\X\times\Y\times\Z$ is said to be a saddle point to the Lagrangian function if it is a solution to the following Karush-Kuhn-Tucker (KKT) system
\begin{equation}\label{optsol}
\A{x}\in\partial f(y), \quad \B x\in\partial g(z)\quad\mbox{and} \quad \A^*y+\B^*z=c.
\end{equation}
Throughout this paper, we make the following assumption which ensures $(\bar y,\bar z)$ is an optimal solution to problem \eqref{prob} and $\bar x$ is an optimal solution to the dual of problem  \eqref{prob}.

\begin{assumption}
\label{assum1}
There exists a vector $(\bar x,\bar y,\bar z)\in\X\times(\dom\,f)\times(\dom\, g)$ satisfying the KKT system \eqref{optsol}.
\end{assumption}

Next, we give a quick review of several ADMM-type methods. In order to improve the performance of the classic ADMM with the step-length $\tau=1$, Eckstein and Bertsekas \cite{ECKSTEIN92} proposed the following generalized ADMM\footnote{In \cite{ECKSTEIN92}, the authors considered the case that $\B^{*}$ is an identity operator and $c=0$. However, it is easy to see that the scheme described in \eqref{gadmm} is a direct application of their idea  to problem \eqref{prob}. } scheme:
\begin{equation}
\label{gadmm}
\left\{
\begin{array}{ll}
y^{k+1}:&\displaystyle=\argmin_{y}
\cL_{\sigma}(y,z^{k};x^{k}),
\\[2mm]
z^{k+1}:&\displaystyle=\argmin_{z}\big\{g(z)-\langle z,\B x^k\rangle+\frac{\sigma}{2}\big\|\rho\A^*y^{k+1}-(1-\rho)\B^*z^k+\B^*z-\rho c\big\|^2\big\},
\\[3mm]
x^{k+1}:&\displaystyle=x^{k}-\sigma \big(\rho\A^*y^{k+1}-(1-\rho)\B^*z^k+\B^*z^{k+1}-\rho c\big),
\end{array}
\right.
\end{equation}
where $\rho\in(0,2)$ is the relaxation factor. For $\rho=1$, the above generalized ADMM scheme is exactly the classic ADMM scheme \eqref{gadmm} with $\tau=1$. In his Ph.D. thesis \cite[Section 3.2]{CHENPHD}, Chen has made the very interesting observation that, cyclically, the generalized ADMM scheme \eqref{gadmm} is equivalent to the following ADMM scheme with initial point $\tilde w^{0}=(\tilde x^{0},\tilde y^{0},\tilde z^{0})\in\X\times(\dom\, f)\times(\dom\, g)$,
\begin{equation}
\label{gadmm2}
\left\{
\begin{array}{l}\displaystyle
z^{k}:=\argmin_{z}
\cL_{\sigma}(\tilde y^{k},z; \tilde x^{k}),
\\[2mm]\displaystyle
x^{k}:=\tilde x^{k}-\sigma(\A^{*} \tilde y^{k}+B^{*} z^{k}-c),
\\[2mm]\displaystyle
y^{k}:=\argmin_{y}\cL_{\sigma}(y,z^{k}; x^{k}),
\\[2mm]\displaystyle
\tilde w^{k+1}:=\tilde w^{k}+\rho(w^{k}-\tilde w^{k}),
\end{array}\right.
\end{equation}
where $\rho\in(0,2)$ is the uniform relaxation factor.
For the subsequent discussion, we need to distinguish between the sequence $\{(x^{k},y^{k},z^{k})\}_{k\ge 0}$ generated by \eqref{gadmm} and the sequence $\{(x^{k},y^{k},z^{k})\}_{k\ge 0}$ generated by \eqref{gadmm2}. For this purpose, we use $\{(\hat x^{k},\hat y^{k},\hat z^{k})\}_{k\ge 0}$ as a surrogate for the sequence $\{(x^{k},y^{k},z^{k})\}_{k\ge 0}$ generated by \eqref{gadmm2}. The equivalence between \eqref{gadmm} and \eqref{gadmm2} is interpreted as follows. Its proof can be found in Chen \cite[Section 3.2]{CHENPHD}.

\begin{proposition}
If the initial point $(x^{0},y^{0},z^{0})\in\X\times(\dom\, f)\times(\dom\, g)$ for the generalized ADMM scheme \eqref{gadmm} satisfies $x^{0}=\hat x^{0}$ and $z^{0}=\hat z^{0}\in\dom\, g$, the sequence $\{(x^{k+1},y^{k+1},z^{k+1})\}_{k\ge 0}$ generated by the generalized ADMM scheme \eqref{gadmm} is exactly the sequence $\{(\hat x^{k+1},\hat y^{k},\hat z^{k+1})\}_{k\ge 0}$ generated by the generalized ADMM scheme \eqref{gadmm2}.
\end{proposition}

The proximal method of multipliers of Rockafellar \cite{rockafellar,ROC78} for solving (\ref{prob}) takes the iterative
scheme
\begin{equation}\label{palmxy}
\displaystyle
(y^{k+1},z^{k+1})=\argmin\limits_{y\in\Y,z\in\Z}\cL_{\sigma}(y,z;x^k)+\frac{1}{2}(\|y-y^k\|_{\cS}^2+\|z-z^k\|_{\T}^2)
\end{equation}
where $\cS$  and $\T$ are the identity operators multiplied by  a positive scalar $\mu$. The proximal method of multipliers is globally convergent under very mild assumptions on the problem's data, namely, convexity and existence of optimal solutions are enough to guarantee convergence. To preserve the good features of both proximal method of multipliers (\ref{palmxy}) and ADMM (\ref{admm}), Eckstein \cite{ECKSTEINOMS} firstly constructed a proximal ADMM, which iterates from the given point $(y^k,z^k;x^k)\in\Y\times\Z\times\X$ in the form of
\begin{subnumcases}{\label{padmm}}
y^{k+1}=\argmin_{y\in\Y}\cL_{\sigma}(y,z^k;x^k)+\frac{1}{2}\|y-y^k\|_{\cS}^2,\label{padmmx}\\
z^{k+1}=\argmin_{z\in\Z}\cL_{\sigma}(y^{k+1},z;x^k)+\frac{1}{2}\|z-z^k\|_{\T}^2,\label{padmmy}\\
x^{k+1}=x^k-\tau\sigma(\A^* y^{k+1}+\B^* z^{k+1}-c),\label{padmml}
\end{subnumcases}
where $\tau=1$, $\cS=\mu_1\I$ and $\T=\mu_2\I$ with $\mu_1>0$ and $\mu_2>0$.

It should be emphasised that the above mentioned approaches require uniformly the linear operators $\cS$ and $\T$ to either be  zero, or positive definite. However, the restriction on the positive definiteness will exclude some interesting situations such as the classic ADMM.
Moreover, the guiding principle is that $\cS$ and $\T$ should be as `small' as possible while (\ref{padmmx}) and (\ref{padmmy}) are still relatively easy to compute.
The exciting progress in designing convergent and efficient (proximal) ADMM in which $\cS$ and $\T$ are required only to be self-adjoint positive semidefinite  owning to Fazel et al. \cite{SEMP13}. This algorithm framework, which was called smi-proximal ADMM (and was abbreviated as sPADMM) later,  takes the form of \eqref{padmmx} to \eqref{padmml} with $\tau\in(0,\frac{1+\sqrt{5}}{2})$. With properly `chosen' $\cS$ and $\T$, sPADMM has exhibited as a powerful algorithmic tool in solving multi-block linearly constrained linear and quadratic convex semidefinite programming problems \cite{CHENL,XDLIMP,lithesis,YANGSIAM,limajorize} to moderate accuracy.

%%%%%%%%%%%%%%%%%%%%%%%%%%%%%%%%%%%%%%%%%%%%%%%%%%%%%%%%%%%%%%%%%%
\section{A generalized ADMM with Semi-Proximal Terms}
\label{spre2}
%%%%%%%%%%%%%%%%%%%%%%%%%%%%%%%%%%%%%%%%%%%%%%%%%%%%%%%%%%%%%%%%%%

Observe that the subproblems in the generalized ADMM schemes \eqref{gadmm} and \eqref{gadmm2} may not admit solutions
because $\A$ or $\B$ is not assumed to be surjective, which is always the case for multi-block convex conic programming problems considered in \cite{CHENL,XDLIMP,lithesis,YANGSIAM,limajorize}.
One natural way to fix this problem is to add  proximal terms to these subproblems. For this purpose and inspired by the sPADMM, we first choose $\cS:\Y\to\Y$ and $\T:\Z\to\Z$ be two self-adjoint positive semidefinite linear operators.
Note that although the generalized ADMM schemes \eqref{gadmm} and \eqref{gadmm2} are equivalent in certain sense, one can get two different variants of the generalized ADMM with semi-proximal terms.
On one hand, we may add the semi-proximal terms $\frac{1}{2}\|y-y^{k}\|_{\cS}$ and $\frac{1}{2}\|z-z^{k}\|_{\T}^{2}$ to the subproblems for computing $y^{k+1}$ and $z^{k+1}$ in the generalized ADMM scheme \eqref{gadmm}. On the other hand, for the generalized ADMM scheme \eqref{gadmm2}, the more natural choice for the semi-proximal terms to be added will be $\frac{1}{2}\|y-\tilde y^{k}\|_{\cS}$ and $\frac{1}{2}\|z-\tilde z^{k}\|_{\T}^{2}$ to the subproblems for computing  $z^{k}$ and $y^{k}$.
In fact, the former one can be reformulated, in light of \eqref{gadmm2}, as
\begin{equation}
\label{gadmm-gu}
\left\{
\begin{array}{rl}
z^{k}:&\displaystyle=\argmin_{z}
\cL_{\sigma}(\tilde y^{k},z; \tilde x^{k})+\frac{1}{2}\|z-z^{k-1}\|_{\T}^{2},
\\[2mm]
x^{k}:&\displaystyle=\tilde x^{k}-\sigma(\A^{*} \tilde y^{k}+B^{*} z^{k}-c),
\\[2mm]
y^{k}:&\displaystyle=\argmin_{y}\cL_{\sigma}(y,z^{k}; x^{k})+\frac{1}{2}\|y-y^{k-1}\|_{\cS}^2,
\\[2mm]
\tilde w^{k+1}:&\displaystyle=\tilde w^{k}+\rho(w^{k}-\tilde w^{k}),
\end{array}\right.
\end{equation}
where $z^{-1}:=\tilde z^{0}$ and $y^{-1}:=\tilde y^{0}$.
In this paper, for the sake of generality and numerical convenience, we consider the latter variant with only semi-proximal terms.
More preciously, we develop the following generalized ADMM  with semi-proximal terms  (abbr. GADMM) in which the most recent values of the variables are used in the proximal terms.
%More specifically, we shall prove the  convergence of the following generalized ADMM\ with semi-proxial terms under the conditions given by Fazel et al. \cite{SEMP13}:

\medskip
\noindent
\begin{minipage}{\textwidth}
\hrule \vskip 1mm
{\renewcommand{\baselinestretch}{1.0}
\begin{algorithm}[GADMM:]
\label{alg1}
{ A generalized ADMM with semi-proximal terms}
\vskip 1.5mm \hrule \vskip 1mm
{\bf Initialization.}
Set $\rho\in(0,2)$ and $\sigma>0$. Choose $\cS:\Y\to\Y$ and $\T:\Z\to\Z$ such that $\Sigma_f+\cS+\A\A^*\succ 0$ and $\Sigma_g+\T+\B\B^*\succ 0$.
Choose an initial point $(\tilde x^{0},\tilde y^{0},\tilde z^{0})\in\X\times\dom(f)\times\dom(g)$. For $k=0,1,\ldots$,
 \\
{\bf Step 1 (main step).} Compute
%%%%%%%%%%%%%%%%$ y
\begin{subnumcases}{\label{pre}}
y^{k}:=\argmin_{y}
\cL_{\sigma}(y,\tilde z^{k};\tilde x^{k})
+\frac{1}{2}\|y-\tilde y^{k}\|^{2}_{\cS},
\label{prex}\\
%%%%%%%%%%%%%%%%% x
x^{k}:=\tilde x^{k}
-\sigma(\A^{*} y^{k}+\B^{*}\tilde z^{k}-c),
\label{prel}\\
%%%%%%%%%%%%%%%%% z
z^{k}:=\argmin_{z}
\cL_{\sigma}(y^{k},z; x^{k})
+\frac{1}{2}\|z-\tilde z^{k}\|^{2}_{\T}
.\label{prey}
\end{subnumcases}

{\bf Step 2 (relaxation step).}  Compute
\begin{equation}\label{relax}
\tilde w^{k+1}:=\tilde w^{k}+\rho(w^{k}-\tilde w^{k}).
\end{equation}
\end{algorithm}
\vskip -3mm }

\medskip
 \hrule
  \vskip 1 mm
 {\renewcommand{\baselinestretch}{1.0} \small
 {\quad} }
 \end{minipage}
\begin{remark}
The order of variable updating in Algorithm \ref{alg1} is different from that used in \eqref{gadmm2}. This swapping is just for the convenience of convergence analysis and essentially does not affect anything else.

\end{remark}
\begin{remark}
Compared with the scheme \eqref{gadmm-gu}, the semi-proximal terms in \eqref{prex} and \eqref{prey} are more natural in the sense that the most recently updated values of variables are involved.
\end{remark}

\begin{remark}
The main objective of this paper is to solve problem (\ref{gmodel}), where $f_2$ and $g_2$ are convex quadratic functions. Therefore, the majorization technique of Li et al. \cite{limajorize}, which can be used together with indefinite $\cS$ and $\T$, is not necessary in this paper. On the other hand, we are more interested in the techniques which can be used to solve multi-block problems. As a result, we only need $\cS$ and $\T$ being positive semidefinite, despite the fact that they can be chosen as indefinite for certain cases. In fact, even $f_2$ and $g_2$ in (\ref{gmodel}) are not separable, a slight modification of our proposed algorithm is applicable by using the techniques in \cite{CUIST}. We keep these techniques away from this paper just for the ease of reading, although the convergence analysis for a unified algorithm with all of these techniques contained would not bring too much work.
\end{remark}

%%%%%%%%%%%%%%%%%%%%%%%%%%%%%%%%%%%%%%%%%%%%%%%%%%%%%%
\section{Applications to Multi-Block Problems}\label{appl}
%%%%%%%%%%%%%%%%%%%%%%%%%%%%%%%%%%%%%%%%%%%%%%%%%%%%%%%%%%%%%%%%%%%%%%%%%%%%%%%%%%%%
In this part, we discuss how to apply the GADMM, i.e, Algorithm \ref{alg1} to solve multi-block separable convex minimization problems.
For the further discussions, we use the notations that $y_{< i}:=(y_1,\ldots,y_{i-1})$, $y_{>i}:=(y_{i+1},\ldots,y_p)$ for $y\in\Y$ and similar notations for $z\in\Z$.

\subsection{The Block Symmetric Gauss-Sidel Iteration Case}
 Firstly, we consider the general convex composite quadratic optimization model \eqref{gmodel}. %, which contains at most two nonsmooth terms in the objective function.:
%\begin{equation}\label{probq}
%\min_{y\in\Y,z\in\Z}\Big\{f_1(y_1)+f_2(y)+g_1(z_1)+g_2(y)\quad\mbox{s.t.}\quad\A^* y+\B^* z=c\Big\},
%\end{equation}
%where $f_1:\Y_1\rightarrow(-\infty,+\infty]$ and $g_1:\Z_1\rightarrow(-\infty,+\infty]$ are closed proper convex functions, and
%$f_2:\Y\rightarrow(-\infty,+\infty)$ and $g_2:\Z\rightarrow(-\infty,+\infty)$ are convex quadratic functions whose
We denote the Hessians of the quadratic functions $f_2$ and $g_2$ by  $\cP$ and $\Q$, respectively,  which are  self-adjoint positive semidefinite  linear operators defined on $\Y$ and $\Z$. According the block sturcture of the variable $y$ and $z$, we can further write $\cP$ and $\Q$ as
$$
\cP=\left(
  \begin{array}{cccc}
    \mathcal{P}_{11}         & \mathcal{P}_{12}     &   \cdots     &\mathcal{P}_{1p}\\
    \mathcal{P}^*_{12}       &\mathcal{P}_{22}      &  \cdots      &\mathcal{P}_{2p} \\
      \vdots                 & \vdots               &  \ddots            &\vdots\\
      \mathcal{P}^*_{1p}     &\mathcal{P}^*_{2p}      &  \cdots      &\mathcal{P}_{pp}
  \end{array}
\right)
\quad
\mbox{and}
\quad
\Q=\left(
  \begin{array}{cccc}
    \mathcal{Q}_{11}         & \mathcal{Q}_{12}     &   \cdots     &\mathcal{Q}_{1q}\\
    \mathcal{Q}^*_{12}       &\mathcal{Q}_{22}      &  \cdots      &\mathcal{Q}_{2q} \\
      \vdots                 & \vdots               &  \ddots            &\vdots\\
      \mathcal{Q}^*_{1q}     &\mathcal{Q}^*_{2q}      &  \cdots      &\mathcal{Q}_{qq}
  \end{array}
\right).
$$
%Let $\sigma>0$ be given. The augmented Lagrangian function of (\ref{probq}) is defined as: for any
%$(x,y,z)\in\X\times\Y\times \Z$,
%$$
%\cL_{\sigma}(y,z;x):=f_1(y_1)+f_2(y)+g_1(z_1)+g_2(z)-\langle x,\A^*y+\B^{*}z-c\rangle+\frac{\sigma}{2}\|\A^*y+\B^*z-c\|^2.
%$$
Moreover, we choose $\E_{i}$ and $\cH_{j}$  to be self-adjoint positive semidefinite linear operators on $\Y_i$ and $\Z_j$ for $i=1,\ldots,p$ and $j=1,\ldots,q$, respectively, such that
$$
\E_{i}+\sigma^{-1}\cP_{ii}+\A_i\A_i^*\succ 0,\quad\mbox{and}\quad\cH_{j}+\sigma^{-1}\Q_{jj}+\B_j\B_j^*\succ 0.
$$
By adopting the smart symmetric Gauss-Seidel (sGS) technique invented by Li et al. \cite{XDLIMP}, the main step (Step $1$) of Algorithm \ref{alg1} can be implemented step by step according to the following procedures.

\medskip
\noindent
\begin{minipage}{\textwidth}
\hrule \vskip 1mm
{\renewcommand{\baselinestretch}{1.0}
\begin{algorithm}\label{alg2}
{The main step of Algorithm \ref{alg1} based on block sGS}
\vskip 1.5mm \hrule \vskip 1mm
{\bf Step 1.1} Compute
\begin{equation}
\left\{
\begin{array}{l}\displaystyle
y^{k-\frac{1}{2}}_i:=\argmin_{y_i} \cL_{\sigma}((\tilde y^{k}_{< i},y_i,y^{k-\frac{1}{2}}_{> i}),\tilde z^{k};\tilde x^{k})+\frac{\sigma}{2}\|y_i-\tilde y_i^k\|_{\E_i}^2,\  i=p,\ldots,2,\\\displaystyle
y^{k}_1:=\argmin_{y_1} \cL_{\sigma}((y_1,y^{k-\frac{1}{2}}_{>1}),\tilde z^{k};\tilde x^{k})+\frac{\sigma}{2}\|y_1-\tilde y_1^k\|_{\E_1}^2,\\\displaystyle
y^{k}_i:=\argmin_{y_i} \cL_{\sigma}((y^{k}_{< i},y_i,y^{k-\frac{1}{2}}_{> i}),\tilde z^{k};\tilde x^{k})+\frac{\sigma}{2}\|y_i-\tilde y_i^k\|_{\E_i}^2, \  i=2,\ldots,p.
\end{array}
\right.
\label{sgsy}
\end{equation}
{\bf Step 1.2} Compute
\begin{equation}
\label{sgs-x}
x^{k}:=\tilde x^{k} -\sigma(\A^{*} y^{k}+\B^{*}\tilde z^{k}-c).
\end{equation}
{\bf Step 1.3} Compute
\begin{equation}
\label{sgsz}
\left\{
\begin{array}{l}\displaystyle
z^{k-\frac{1}{2}}_j:=\argmin_{z_j} \cL_{\sigma}(y^{k},(\tilde z^{k}_{< j},z_j,z^{k-\frac{1}{2}}_{> j}); x^{k})+\frac{\sigma}{2}\|z_j-\tilde z_j^k\|^2_{\cH_{j}},\  j=q,\ldots,2,\\ \displaystyle
z^{k}_1:=\argmin_{z_1} \cL_{\sigma}(y^{k},(z_1,z^{k-\frac{1}{2}}_{>1}); x^{k})+\frac{\sigma}{2}\|z_1-\tilde z_1^k\|_{\cH_1}^2,\\ \displaystyle
z^{k}_j:=\argmin_{z_j} \cL_{\sigma}(y^{k},(z^{k}_{< j},z_j,z^{k-\frac{1}{2}}_{> j}); x^{k})+\frac{\sigma}{2}\|z_j-\tilde z_j^k\|^2_{\cH_{j}},\  j=2,\ldots,q.
\end{array}
\right.
\end{equation}
\end{algorithm}
\vskip -3mm }

\medskip
 \hrule
  \vskip 1 mm
 {\renewcommand{\baselinestretch}{1.0} \small
 {\quad} }
 \end{minipage}

\medskip
Next, denote $\D_y:=\diag(\E_{1}+\sigma^{-1}\cP_{11}+\A_1\A_1^*,\ldots,\E_{p}+\sigma^{-1}\cP_{pp}+\A_p\A_p^*)$ and $\D_z:=\diag(\cH_{1}+\sigma^{-1}\Q_{11}+\B_1\B_1^*,\ldots,\cH_{q}+\sigma^{-1}\Q_{qq}+\B_q\B_q^*)$. Furthermore, define the following linear operators:
$$
\M_y:=\left(
  \begin{array}{cccc}
    0        &\sigma^{-1}\cP_{12}+\A_1\A_2^*     &\cdots    &\sigma^{-1}\cP_{1p}+\A_1\A_p^*\\[3mm]
             &\ddots                            &\cdots    &\vdots\\[3mm]  %µÚ¶þÐÐÔªËØ
             &                                  &0         &\sigma^{-1}\cP_{(p-1)p}+\A_{p-1}\A_p^*\\[3mm]
             &                                  &          & 0
  \end{array}
\right)
$$
and
$$
\M_z:=\left(
  \begin{array}{cccc}
    0        &\sigma^{-1}\Q_{12}+\B_1\B_2^*     &\cdots    &\sigma^{-1}\Q_{1q}+\B_1\B_q^*\\[3mm]
             &\ddots                            &\cdots    &\vdots\\[3mm]
             &                                  &0         &\sigma^{-1}\Q_{(q-1)q}+\B_{q-1}\B_q^*\\[3mm]
             &                                  &          & 0
  \end{array}
\right).
$$
From  Li et al. \cite{XDLIMP}, it is not hard to see that the point $(z^k,y^k,x^k)$ produced by the precedure (\ref{sgsy}), \eqref{sgs-x} and (\ref{sgsz}) is exactly that generated by (\ref{prex}) to (\ref{prey}) with the two positive semidefinite linear operators
$$
\cS=\sigma(\M_y\D_y^{-1}\M_y^*+\E)\quad\mbox{and}\quad \T=\sigma(\M_z\D_z^{-1}\M_z^*+\cH),
$$
where $\E=\diag(\E_1,\ldots,\E_p)$, and $\cH=\diag(\cH_1,\ldots,\cH_q)$.

\subsection{The Full Jacobian Iteration Case}
Now we turn our attention to a much general case that there are at least three nonsmooth blocks in the objective function of problem (\ref{prob}), i.e.,
$$
f(y)=\sum_{i=1}^{p}f_i(y_i), \quad\mbox{and}\quad g(z)=\sum_{j=1}^{q}g_j(z_j),
$$
where $f_i:\Y_i \rightarrow(-\infty,+\infty]$ and $g_j:\Z_j \rightarrow(-\infty,+\infty]$ are closed proper convex (probably nonsmooth) functions.
The convergence of the multi-block ADMM in a Gauss-Seidel manner for solving such problems under certain further conditions has been considered, e.g. \cite{MASIAMOPT,MACOAP}, but we are interested in removing these conditions, e.g., strong convexity, continuously differentiable with Lipschitz continuous gradients etc..
 Let $\tau_1$ and $\tau_2$ be two given numbers such that $\tau_1\geq p-1$ and $ \tau_2\geq q-1$.
One can choose $\E_{i}$ and $\cH_{j}$  to be self-adjoint positive semidefinite linear operators on $\Y_i$ and $\Z_j$ for $i=1,\ldots,p$ and $j=1,\ldots,q$, respectively, such that for any $1\le i\le p$ and $1\le j\le q$,
$$
\E_{i}\succeq \tau_1\A_i\A_i^*,\  \Sigma_{f_i}+\E_i+\A_i\A_i^*\succ 0\quad\mbox{and}\quad \cH_{j}\succeq \tau_2\B_j\B_j^*,\  \Sigma_{g_j}+\cH_j+\B_j\B_j^*\succ 0.
$$
Define $\E:=\diag(\E_1,\ldots\E_p)$ and $\cH:=\diag(\cH_1,\ldots\cH_q)$ being two block-diagonal operators and let
\begin{equation}\label{sstt}
\cS:=\sigma[(1+\tau_1)\E-\A\A^*] \quad\mbox{and}\quad
\T:=\sigma[(1+\tau_2)-\B\B^*].
\end{equation}
It is easy to verify that $\cS\succeq 0$ and $\T\succeq 0$, while it holds that $\Sigma_f+\A\A^*+\E\succ 0$ and $\Sigma_g+\B\B^*+\cH\succ 0$.

Now we consider the following procedure that generates $\omega^k=(y^k,z^k,x^k)$ and $\tilde\omega^{k+1}$ in a parallel manner from the given point $\tilde\omega^k=(\tilde y^k,\tilde z^k,\tilde x^k)$, i.e.,
\[\left\{
\begin{array}{l}\displaystyle
y^k_i=\argmin_{x_i\in\X_i}\cL_{\sigma}\big((\tilde y^k_{<i},y_i,\tilde y^k_{>i}),\tilde z^k;\tilde x^k\big)+\frac{\sigma}{2}\|y_i-\tilde y_i^k\|_{\E_i}^2,\quad 1\le i\le p,\\[2mm]\displaystyle
x^k=\tilde x^k-\sigma(\A^*y^k+\B^*\tilde z^k-c),\\[2mm]\displaystyle
z^k_j=\argmin_{z\in\Z_j}\cL_{\sigma}\big(y^k,(\tilde z^k_{<j},z_i,\tilde z^k_{>j});x^k\big)+\frac{\sigma}{2}\|z_j-\tilde z^k_j\|_{\cH_j}^2,\quad 1\le j\le q,\\[2mm]\displaystyle
\tilde \omega^{k+1}=\tilde \omega^k+\rho(\omega^k-\tilde \omega^k).
\end{array}
\right.\]
It is straightforward to deduce that the above iterative scheme is exactly an instance of the Algorithm \ref{alg1} with $\cS$ and $\T$ defined by (\ref{sstt}).

%%%%%%%%%%%%%%%%%%%%%%%%%%%%%%%%%%%%%%%%%%%%%%%%%%%%%%
\section{Convergence Analysis}
\label{glob}
%%%%%%%%%%%%%%%%%%%%%%%%%%%%%%%%%%%%%%%%%%%%%%%%%%%%%%%%%%%%%
In this section, we analyze the convergence of Algorithm \ref{alg1} step by step.
For this purpose, we need the following two basic equalities:
\begin{equation}
\label{ele1}
2\langle u_1-u_2,\mathcal{G}(v_1-v_2)\rangle =\|u_1-v_2\|^2_\mathcal{G}-\|u_1-v_1\|^2_\mathcal{G}+\|u_2-v_1\|^2_\mathcal{G}-\|u_2-v_2\|^2_\mathcal{G},
\end{equation}
and
\begin{equation}
\label{ele2}
2\langle u_1,\mathcal{G} u_2\rangle=\|u_1\|^2_\mathcal{G}+\|u_2\|^2_\mathcal{G}-\|u_1-u_2\|^2_\mathcal{G}=\|u_1+u_2\|^2_\mathcal{G}-\|u_1\|^2_\mathcal{G}-\|u_2\|^2_\mathcal{G},
\end{equation}
where $u_1$, $u_2$, $v_1$, and $v_2$ are vectors in the same finite dimensional real Euclidean space endowed with inner product $\langle\cdot,\cdot\rangle$ and induced norm $\|\cdot\|$, and $\mathcal{G}$ is an arbitrary self-adjoint positive semidefinite linear operator from that space to itself.

Next, let $(\bar x,\bar y,\bar z)\in\X\times\Y\times\Z$ be an arbitrary solution to the KKT system \eqref{optsol}. For any $(x,y,z)\in\X\times\Y\times\Z$ we denote $x_{e}=x-\bar x$, $y_{e}=y-\bar y$ and $z_{e}=z-\bar z$. We first have the following result.
\begin{lemma}
\label{lemma1}
Let $(\bar x,\bar y,\bar z)$ be a solution to the KKT system \eqref{optsol} and $\{(x^{k},y^k,z^k)\}$ be the sequence generated by Algorithm \ref{alg1}. For any $k\ge 0$ it holds that
\begin{equation}
\label{rlem1}
\langle x^{k+1}-x^k,\A^*(y^{k+1}-y^k)\rangle
-\frac{\rho}{2}\|y^{k+1}-\tilde{y}^{k+1}\|^2_\cS+\frac{\rho}{2}\|y^k-\tilde{y}^k\|^2_\cS\ge\|y^{k+1}-y^k\|_{\Sigma_f}^2,
\end{equation}
\medskip
\begin{equation}\label{firstterm}
\begin{array}{ll}
\frac{\sigma\rho}{2}\|\A^*y^{k+1}_e+\B^*z^k_e\|^2+ \langle x^{k+1}_e+\sigma(1-\rho)\A^*y^{k+1}_e,\A^*y^{k+1}_e+\B^*z^k_e\rangle
\\[2mm]
\ =-\frac{1}{2\sigma\rho}\Big[\|x^{k+1}_e+\sigma(1-\rho)\A^*y^{k+1}_e\|^2
-\|x^k_e+\sigma(1-\rho)\A^*y^k_e\|^2\Big]
\end{array}
\end{equation}
and
\begin{equation}\label{345term}
\begin{array}{lll}
\sigma(1-\rho)\langle\B^*z^k_e,\A^*y^{k+1}_e+\B^*z^k_e\rangle
 +\langle x^k-x^{k+1},\B^*z^k_e\rangle
-\sigma\langle\A^*y^k+\B^*z^k-c,\B^*z^k_e\rangle
\\[2mm]
\leq
\frac{2-\rho}{\rho}\Big[-\frac{\rho}{2}\|y^{k+1}-\tilde{y}^{k+1}\|^2_\cS+\frac{\rho}{2}\|y^k-\tilde{y}^k\|^2_\cS- \|y^{k+1}-y^k\|_{\Sigma_f}^2\\[2mm]
\qquad\qquad+\frac{\sigma\rho}{2}\|\A^*y^k_e\|^2-\frac{\sigma\rho}{2}\|\A^*y^{k+1}_e\|^2
-\frac{\sigma(2-\rho)}{2}\|\A^*y^k_e-\A^*y^{k+1}_e\|^2\Big].
\end{array}
\end{equation}

\end{lemma}

\begin{proof}
By the first order optimality condition of (\ref{prex}), we get
\begin{equation}
\label{optx}
\A\tilde{x}^k-\sigma \A(\A^*y^k+\B^*\tilde{z}^k-c)-\cS(y^k-\tilde{y}^k)\in\partial f(y^k).
\end{equation}
It follows from (\ref{prel}) that
\begin{equation}
\label{optx2}
\A x^k-\cS(y^k-\tilde{y}^k)\in\partial f(y^k),\quad\mbox{and}\quad \A x^{k+1}-\cS(y^{k+1}-\tilde{y}^{k+1})\in\partial f(y^{k+1}),
\end{equation}
which, together with (\ref{ineqx}), implies
\begin{equation}
\label{firstx}
\begin{array}{l}
\langle x^{k+1}-x^k,\A^*(y^{k+1}-y^k)\rangle
-\langle\cS(y^{k+1}-\tilde{y}^{k+1})-\cS(y^k-\tilde{y}^k),y^{k+1}-y^k\rangle
\\[2mm]
\geq \|y^{k+1}-y^k\|_{\Sigma_f}^2.
\end{array}
\end{equation}
From (\ref{relax}), we can easily deduce
\begin{equation}
\begin{array}{ll}
 &\langle\cS(y^{k+1}-\tilde{y}^{k+1})-\cS(y^k-\tilde{y}^k),y^{k+1}-y^k\rangle\\[2mm]
=&\langle\cS(y^{k+1}-\tilde{y}^{k+1})-\cS(y^k-\tilde{y}^k),(y^{k+1}-\tilde{y}^{k+1})+(\tilde{y}^{k+1}-y^k)\rangle\\[2mm]
=&\langle\cS(y^{k+1}-\tilde{y}^{k+1})-\cS(y^k-\tilde{y}^k),(y^{k+1}-\tilde{y}^{k+1})-(1-\rho)(y^k-\tilde{y}^k)\rangle\\[2mm]
=&\|y^{k+1}-\tilde{y}^{k+1}\|^2_\cS-(\rho-1)\|y^k
-\tilde{y}^k\|^2_\cS-(2-\rho)\langle\cS(y^{k+1}
-\tilde{y}^{k+1}),y^k-\tilde{y}^k\rangle\\[2mm]
\ge
&\|y^{k+1}-\tilde{y}^{k+1}\|^2_\cS-(\rho-1)\|y^k
-\tilde{y}^k\|^2_\cS
-\frac{2-\rho}{2}(\|y^{k+1}-\tilde{y}^{k+1}\|^2_\cS
+\|y^k-\tilde{y}^k\|^2_\cS)
\\[2mm]
=&\frac{\rho}{2}\|y^{k+1}
-\tilde{y}^{k+1}\|^2_\cS
-\frac{\rho}{2}\|y^k
-\tilde{y}^k\|^2_\cS,
\end{array}
\end{equation}
which, together with the inequality (\ref{firstx}), implies (\ref{rlem1}).

Secondly, we get from (\ref{prel}) and (\ref{relax}) that
\begin{equation}
\label{llkk}
\tilde{x}^{k+1}=x^{k}+(\rho-1)(x^{k}-\tilde{x}^k)=x^{k}-\sigma(\rho-1)(\A^*y^k+\B^*\tilde{z}^{k}-c).
\end{equation}
From (\ref{prel}), we can get
$$
\begin{array}{rl}
x^{k+1}= &\tilde{x}^{k+1}-\sigma(\A^*y^{k+1}+\B^* \tilde{z}^{k+1}-c),\\
=&x^{k}-\sigma\rho(\A^*y^{k+1}+\B^*z^{k}-c)-\sigma(\rho-1)(\A^*y^k-\A^*y^{k+1}).
\end{array}
$$
This implies
\begin{equation}
\label{ut171}
[x^{k+1}_e+\sigma(1-\rho)\A^*y^{k+1}_e]-[x^k_e+\sigma(1-\rho)\A^*y^k_e]=-\sigma\rho(\A^*y^{k+1}_e+\B^*z^{k}_e).
\end{equation}
Subsequently, using the relation (\ref{ut171}) and the elementary equality (\ref{ele2}), we can get \eqref{firstterm}.

Finally, from \eqref{llkk} we can deduce
\begin{equation}\begin{array}{lll}
&\sigma(1-\rho)(\A^*y^{k+1}_e+\B^*z^k_e)+(x^k-x^{k+1})-\sigma(\A^*y^k+\B^*z^k-c)
\\[2mm]
=&\sigma(1-\rho)(\A^*y^{k+1}_e+\B^*z^k_e)
+\sigma\rho(\A^*y^{k+1}_e+\B^*z^{k}_e)
\\[1.5mm]
&+\sigma(\rho-1)(\A^*y^k_e-\A^*y^{k+1}_e)-\sigma(\A^*y^k_e+\B^*z^k_e)
\\[2mm]
=&-\sigma(2-\rho)(\A^*y^k_e-\A^*y^{k+1}_e).
\end{array}\end{equation}
Thus, we have
\begin{equation}\label{utev171}
\begin{array}{l}
\sigma(1-\rho)\langle\A^*y^{k+1}_e+\B^*z^k_e,\B^*z^k_e\rangle
+\langle x^k-x^{k+1},\B^*z^k_e\rangle
\\[1.5mm]
\quad-\sigma\langle\A^*y^k+\B^*z^k-c,\B^*z^k_e\rangle
\\[2mm]
\quad=-\sigma(2-\rho)\langle\A^*y^k_e-\A^*y^{k+1}_e,\B^*z^k_e\rangle
\\[2mm]
\quad=-(2-\rho)\sigma\langle\A^*y^k_e-\A^*y^{k+1}_e,\A^*y^{k+1}_e+\B^*z^k_e\rangle
\\[1.5mm]
\quad\quad+(2-\rho)\sigma\langle\A^*y^k_e-\A^*y^{k+1}_e,\A^*y^{k+1}_e\rangle.
\end{array}
\end{equation}
By (\ref{ut171}) we have
\begin{equation}
\begin{array}{lll}
 &\sigma\langle\A^*y^k_e-\A^*y^{k+1}_e,\A^*y^{k+1}_e+\B^*z^k_e\rangle\\[2mm]
=&-\rho^{-1}\langle\A^*y^k_e-\A^*y^{k+1}_e,x^{k+1}_e-x^k_e\rangle
+\rho^{-1}\sigma(1-\rho)\|\A^*y^{k+1}_e-\A^*y^k_e\|^2.\label{eight1}
\end{array}
\end{equation}
On the other hand, by using the elementary equality (\ref{ele2}), we can get
\begin{equation}\label{eight2}
\sigma\langle\A^*y^k_e-\A^*y^{k+1}_e,\A^*y^{k+1}_e\rangle
=\frac{\sigma}{2}\Big[\|\A^*y^k_e\|^2-\|\A^*y^k_e-\A^*y^{k+1}_e\|^2
-\|\A^*y^{k+1}_e\|^2 \Big].
\end{equation}
Then by substituting (\ref{eight1}) and (\ref{eight2}) into (\ref{utev171}) and using inequality (\ref{rlem1}) we can get (\ref{345term}).
\qed
\end{proof}

%%%%%%%%%%%%%%%%%[2]%%%%%%%%%%%%%%%%
In what follows, we will establish a key inequality for analyzing the convergence of Algorithm \ref{alg1}. We define the sequence $\{\phi_{k}\}_{k\ge 0}$ by
\begin{equation}
\label{phikk}
\begin{array}{ll}
\phi_{k}:=&
\frac{1}{\sigma\rho}\|x^k_e+\sigma(1-\rho)\A^*y^k_e\|^2
+\rho^{-1}\big[\|\tilde{y}^{k+1}_{e}\|^2_\cS
+\|\tilde{z}^k_{e}\|^2_\T\big]
\\[2mm]
&+(2-\rho)\big[\|y^k-\tilde{y}^k\|^2_\cS+\sigma\|\A^*y^k_e\|^2\big].
\end{array}
\end{equation}

\begin{lemma}\label{lemma5}
Suppose Assumption \ref{assum1} holds. Let the sequence $\{(x^{k},y^k,z^k)\}$ be generated by Algorithm \ref{alg1}. Then for any $k\ge 0$, it holds that
\begin{equation}
\label{conclusion:inequality}
\begin{array}{rl}
\phi_{k}-\phi_{k+1}\ge&
2\|y^{k+1}_e\|^2_{\Sigma_f}+2\|z^k_e\|^2_{\Sigma_g}+\sigma(2-\rho)\|\A^*y^{k+1}_e+\B^*z^k_e\|^2
\\[2mm]
&+(2-\rho)\|\tilde{y}^{k+1}-y^{k+1}\|^2_\cS
+(2-\rho)\|\tilde{z}^k-z^k\|^2_\T
\\[2mm]
&+2(2-\rho)\rho^{-1}\|y^{k+1}-y^k\|_{\Sigma_f}^2+\sigma\rho^{-1}(2-\rho)^2\|\A^*y^k_e-\A^*y^{k+1}_e\|^2.
\end{array}
\end{equation}
\end{lemma}

\begin{proof}
We get from (\ref{prey})
\begin{equation}
\label{opty}
\B x^{k+1}+(\B x^k-\B x^{k+1})-\sigma \B(\A^* y^k+\B^*z^k-c)-\T(z^k-\tilde{z}^k)\in\partial g(z^k).
\end{equation}
By combining (\ref{opty}), (\ref{optsol}) and (\ref{ineqx}) together, one obtains
\begin{equation}
\label{opty3}
\begin{array}{l}
\langle x^{k+1}_e,\B^*z^k_e\rangle+\langle x^k-x^{k+1},\B^*z^k_e\rangle
-\langle\T(z^k-\tilde{z}^k),z^k_e\rangle
\\[2mm]
-\sigma\langle\A^*y^k+\B^*z^k-c,\B^*z^k_e\rangle
\geq \|z^k_e\|^2_{\Sigma_g}.
\end{array}
\end{equation}
%%%%%%%%%%%%%%%%%%%%%%%%%%%%%%%
On the other hand, by combining (\ref{optx2}), (\ref{optsol}) and (\ref{ineqx}) together, we obtain
\begin{equation}\label{optx3}
\langle x^{k+1}_e,\A^*y^{k+1}_e\rangle-\langle\cS(y^{k+1}-\tilde{y}^{k+1}),y^{k+1}_e\rangle
\geq \|y^{k+1}_e\|^2_{\Sigma_f}.
\end{equation}
Then by adding (\ref{opty3}) to (\ref{optx3}) we get
%%%%%%%%%%%%%%
\begin{equation}
\label{rlem2}
\begin{array}{l}
\langle x^{k+1}_e,\A^*y^{k+1}_e+\B^*z^k_e\rangle+\langle x^k-x^{k+1},\B^*z^k_e\rangle
-\sigma\langle\A^*y^k+\B^*z^k-c,\B^*z^k_e\rangle
\\[2mm]
-\langle\cS(y^{k+1}-\tilde{y}^{k+1}),y^{k+1}_e\rangle
-\langle\T(z^k-\tilde{z}^k),z^k_e\rangle\geq \|y^{k+1}_e\|^2_{\Sigma_f}+\|z^k_e\|^2_{\Sigma_g}.
\end{array}
\end{equation}
Note that
\begin{equation}
\label{estimate1}
\begin{array}{ll}
&\langle x^{k+1}_e,\A^*y^{k+1}_e+\B^*z^k_e\rangle
\\[2mm]
=&\langle x^{k+1}_e+\sigma(1-\rho)\A^*y^{k+1}_e,\A^* y^{k+1}_e+\B^*z^k_e\rangle
\\[1mm]
& -\sigma(1-\rho)\langle\A^* y^{k+1}_e,\A^*y^{k+1}_e+\B^*z^k_e\rangle
\\[2mm]
=&\langle x^{k+1}_e+\sigma(1-\rho)\A^*y^{k+1}_e,\A^*y^{k+1}_e+\B^*z^k_e\rangle
\\[1mm]
& -\sigma(1-\rho)\langle\A^*y^{k+1}_e+\B^*z^k_e,\A^*y^{k+1}_e+\B^*z^k_e\rangle
 \\[1mm]
 &+\sigma(1-\rho)\langle\B^*z^k_e,\A^*y^{k+1}_e+\B^*z^k_e\rangle.
\end{array}
\end{equation}
Consequently, \eqref{rlem2} can be reformulated as
\begin{equation}
\label{lem2}
\begin{array}{ll}
\langle x^{k+1}_e+\sigma(1-\rho)\A^*y^{k+1}_e,\A^*y^{k+1}_e+\B^*z^k_e\rangle
-\sigma(1-\rho)\|\A^*y^{k+1}_e+\B^*z^k_e\|^2
\\[2mm]
+\sigma(1-\rho)\langle\A^*y^{k+1}_e+\B^*z^k_e,\B^*z^k_e\rangle
+\langle x^k-x^{k+1},\B^*z^k_e\rangle
\\[2mm]
-\sigma\langle\A^*y^k+\B^*z^k-c,\B^*z^k_e\rangle
-\langle\cS(y^{k+1}-\tilde{y}^{k+1}),y^{k+1}_e\rangle
\\[2mm]
-\langle\T(z^k-\tilde{z}^k),z^k_e\rangle\geq \|y^{k+1}_e\|^2_{\Sigma_f}+\|z^k_e\|^2_{\Sigma_g}.
\end{array}
\end{equation}
Next, we shall estimate the left-hand-side of (\ref{lem2}).
Using the relation (\ref{relax}) and the elementary equality (\ref{ele1}), we get
\begin{equation}\label{sxk}
\langle\cS(y^{k+1}-\tilde{y}^{k+1}),y^{k+1}_e\rangle
=\frac{1}{2\rho}\Big[\|\tilde{y}^{k+2}-\bar{y}\|^2_\cS-\|\tilde{y}^{k+1}-\bar{y}\|^2_\cS+\rho(2-\rho)\|\tilde{y}^{k+1}-y^{k+1}\|^2_\cS\Big],
\end{equation}
and
\begin{equation}\label{tyk}
\langle\T(z^k-\tilde{z}^k),z^k_e\rangle
=\frac{1}{2\rho}\Big[\|\tilde{z}^{k+1}-\bar{z}\|^2_\T-\|\tilde{z}^k-\bar{z}\|^2_\T+\rho(2-\rho)\|\tilde{z}^k-z^k\|^2_\T\Big].
\end{equation}
Substituting (\ref{firstterm}), (\ref{345term}), (\ref{sxk}), and (\ref{tyk}) into  (\ref{lem2}), we get
\begin{equation}
\label{ineq:des}
\begin{array}{lll}
-\frac{1}{\sigma\rho}\Big[\|x^{k+1}_e
+\sigma(1-\rho)\A^*y^{k+1}_e\|^2
-\|x^k_e+\sigma(1-\rho)\A^*y^k_e\|^2\Big]
\\[2mm]
-\frac{1}{\rho}\Big[\|\tilde{y}^{k+2}-\bar{y}\|^2_\cS
-\|\tilde{y}^{k+1}-\bar{y}\|^2_\cS\Big]
-\frac{1}{\rho}\Big[\|\tilde{z}^{k+1}-\bar{z}\|^2_\T
-\|\tilde{z}^k-\bar{z}\|^2_\T\Big]
\\[2mm]
+(2-\rho)\rho^{-1}\Big[\rho\|y^k-\tilde{y}^k\|^2_\cS
-\rho\|y^{k+1}-\tilde{y}^{k+1}\|^2_\cS
-\sigma\rho\|\A^*y^{k+1}_e\|^2
+\sigma\rho\|\A^*y^k_e\|^2\Big]
\\[3mm]
\geq
2\|y^{k+1}_e\|^2_{\Sigma_f}
+2\|z^k_e\|^2_{\Sigma_g}
+\sigma(2-\rho)\|\A^*y^{k+1}_e+\B^*z^k_e\|^2
\\[2mm]
\quad+(2-\rho)\|\tilde{y}^{k+1}-y^{k+1}\|^2_\cS+(2-\rho)\|\tilde{z}^k-z^k\|^2_\T
\\[2mm]
\quad+2(2-\rho)\rho^{-1}\|y^{k+1}-y^k\|_{\Sigma_f}^2+\sigma\rho^{-1}(2-\rho)^2\|\A^*y^k_e-\A^*y^{k+1}_e\|^2.
\end{array}
\end{equation}
Hence, we get \eqref{conclusion:inequality} by the definition of $\phi_k$ and complete the proof.
\qed
\end{proof}
\begin{remark}
The above proof looks like to those in \cite{SEMP13} and \cite{admmnote} since the basic tools for proving all these theoretical results are closely based on an earlier work of Fortin and Glowinski\cite{FORTIN83}.
However, the most important steps of this paper are quite different from the previous ones.
An explicit example is the definition of $\phi_k$ in \eqref{conclusion:inequality},
which originally comes from \eqref{estimate1}, that plays the key role  to get   \eqref{ineq:des}.
The motivation is that we want to decompose $x_e^{k+1}$ so that there is not any inner product terms in \eqref{conclusion:inequality}.
Therefore, we need some elementary but careful calculations, which are not previously conducted.
\end{remark}

Now, we are ready to establish the global convergence of Algorithm \ref{alg1}.

\begin{theorem}\label{theorem}
Suppose Assumption \ref{assum1} holds. Let $\{(x^k,y^k,z^k)\}$ be the sequence generated by Algorithm \ref{alg1}. Then the whole sequence $\{(x^k,y^k,z^k)\}$ converges to a solution to the KKT system \eqref{optsol}.
\end{theorem}
\begin{proof}
Note that $\rho\in(0,2)$. We clearly see from \eqref{conclusion:inequality} that $\{\phi_k\}_{k\geq 0}$ is a nonnegative and monotonically non-increasing sequence. Hence, $\{\phi_k\}$ is also bounded. As a result, the following sequences are bounded:
\begin{equation}\label{bound}
\{\|x^k_e+\sigma(1-\rho)\A^*y^k_e\|\},\,
\{\|\tilde{y}^{k+1}\|_\cS\},\,
\{\|\tilde{z}^k\|_\T\},\,
\{\|y^k-\tilde{y}^k\|_\cS\},
\, \mbox{and}\, \{\|\A^*y^k_e\|\}.
\end{equation}
Moreover, from the inequality \eqref{conclusion:inequality} we have as $k\to\infty$,
\begin{equation}
\label{mullim}
\begin{array}{l}
\|y^{k+1}_e\|_{\Sigma_f}\to 0,\,
\|z^k_e\|_{\Sigma_g}\rightarrow 0,\,
\|\A^*y^{k+1}_e+\B^*z^k_e\|\rightarrow 0,\,
\|\tilde{y}^{k+1}-y^{k+1}\|_\cS\rightarrow0,
\\[2mm]
\|\tilde{z}^k-z^k\|_\T\rightarrow 0,\,
\|y^{k+1}-y^k\|_{\Sigma_f}\rightarrow 0,\
\mbox{and}\
\|\A^*y^k_e-\A^*y^{k+1}_e\|\rightarrow 0.
\end{array}
\end{equation}
Thus, by the fact $\|y^k\|_\cS\leq\|y^k-\tilde{y}^k\|_\cS+\|\tilde{y}^k\|_\cS$ we can see that $\{\|y^k\|_\cS\}$ is also bounded. Consequently, the sequence $\{\|y^k\|_{\Sigma_f+\cS+\A\A^*}\}$ is bounded. Since $\Sigma_f+\cS+\A\A^*\succ0$, the sequence $\{\|y^k\|\}$ is bounded.
Similarly, the sequences $\{\|z^k\|_\T\}$ and $\{\|\B^*z^k\|\}$ are both bounded and so is the sequence $\{\|z^k\|_{\Sigma_g+\T+\B\B^*}\}$.
It implies that the sequence $\{\|z^k\|\}$ is bounded from $\Sigma_g+\T+\B\B^*\succ 0$.
The boundedness of $\{\|x^k_e+\sigma(1-\rho)\A^*y^k_e\|\}$ and $\{\|y^k\|\}$ further indicate that the sequence $\{\|x^k\|\}$ is bounded. The above arguments have shown that $\{(x^{k},y^{k},z^{k})\}$ is a bounded sequence.

Consequently, the sequence $\{(x^{k},y^{k},z^{k})\}$ admits at least one convergent subsequence. Suppose that $\{(x^{k_i},y^{k_i},z^{k_i})\}$ is a subsequence of $\{(x^{k},y^{k},z^{k})\}$ converging to $(x^\infty,y^\infty,z^\infty)\in\X\times\Y\times\Z$.
It follows from \eqref{optx2} and \eqref{opty} that
\begin{equation}
\label{forcon}
\left\{
\begin{array}{l}
\A x^{k_i}-\cS(y^{k_i}-\tilde{y}^{k_i})\in\partial f(y^{k_i}),\\[2mm]
\B x^{k_i}-\sigma \B(\A^* y^{k_i}+\B^*z^{k_i}-c)-\T(z^{k_i}-\tilde{z}^{k_i})\in\partial g(z^{k_i}).
\end{array}
\right.
\end{equation}
We also have from (\ref{mullim}) that $\lim_{k\to\infty}(\A^{*}y^{k}+\B^{*}z^{k}-c)=0$. Taking limits in (\ref{forcon}) and using (\ref{mullim}), one obtains
$$
\A x^\infty\in\partial f(y^\infty),\quad
\B x^\infty\in\partial g(z^\infty)
\quad\mbox{and}\quad
\A^*y^\infty+\B^*z^\infty-c=0,
$$
which indicates that $(x^\infty,y^\infty,z^\infty)$ is a solution to the KKT system \eqref{optsol}.

Next we need to show that  $(x^\infty,y^\infty,z^\infty)$ is the unique limit point of the sequence $\{(x^{k},y^{k},z^{k})\}$. Without loss of generality we can let $(\bar x,\bar y,\bar z)=(x^\infty,y^\infty,z^\infty)$.
Consequently, the sequence $\{\phi_{k}\}$ defined by \eqref{phikk} has a subsequence converges to zero. Hence $\{\phi_{k}\}$ itself converging to zero. By the definition of $\phi_{k}$ we have $\lim_{k\rightarrow\infty}x^k= \bar{x}$. Moreover, from
$\|\tilde{y}^{k+1}-y^{k+1}\|_\cS\rightarrow0$ in (\ref{mullim}), it is easy to show that $\|y^{k}_e\|_\cS\rightarrow 0$ as $k\rightarrow\infty$. Noting that $\|\A^*y^k_e\|\rightarrow 0$ in (\ref{bound}) and $\|y^k_e\|_{\Sigma_f}\rightarrow 0$
in (\ref{mullim}), we have
$\{\|y^{k}_e\|_{\Sigma_f}+\|y^{k}_e\|_\cS+\|\A^*y^{k}_e\|\}\to 0$ as $k\to\infty$.
Hence, we have $\lim_{k\rightarrow\infty}y^k=\bar{y}$ since $\Sigma_f+\cS+\A\A^*\succ 0$.
Finally, from the fact $\|\A^*y^{k+1}_e\|\rightarrow 0$,  $\|\A^*y^{k+1}_e+\B^*z^k_e\|\rightarrow 0$ in (\ref{mullim}), and
$$
\|\B^*z^k_e\|\leq\|\A^*y^{k+1}_e+\B^*z^k_e\|+\|\A^*y^{k+1}_e\|,
$$
we get $\|\B^*z^k_e\|\rightarrow 0$.
Since $\|z^k_e\|_\T\rightarrow 0$ and $\|z^k_e\|_{\Sigma_g}\rightarrow 0$ by (\ref{mullim}), we have
$\{\|z^{k}_e\|_{\Sigma_g}+\|z^{k}_e\|_\T+\|\B^*z^{k}_e\|\}\to 0$ as $k\to\infty$.
Therefore, from the fact $\Sigma_g+\T+\B\B^*\succ0$ we know that
$\lim_{k\rightarrow\infty}z^k=\bar{z}$.  This completes the proof.
\qed
\end{proof}

%%%%%%%%%%%%%%%%%%%%%%%%%%%%%%%%%%%%%%%%%%%%%%%%%%%%%%%%%%%5
\section{Numerical Experiments}
\label{numer}

In this section, we apply the generalized ADMM with semi-proximal terms (Algorithm \ref{alg1}) proposed in this paper together with the block sGS techniques (Algorithm \ref{alg2}) to solve convex DNN-SDP problem via its dual. Specifically, the DNN-SDP
problem is represented by
\begin{eqnarray}
\begin{array}{ll}
\min\limits_X &  \ds \langle{C},{X}\rangle  \\[2mm]
\mbox{s.t.}
       &\A_E X   =  b_E, \ \A_I X    \geq b_I, \  X \in \cS^n_+\cap \N,
\end{array}
\label{dnnsdp}
\end{eqnarray}
where $\cS^n$ is the space of $n\times n$ symmetric matrices,
$\A_E:\cS^n \rightarrow \Re^{m_E}$ and $\A_I:\cS^n\rightarrow \Re^{m_I}$ are two linear maps, $\cS_+^n$ is the closed convex cone of $n\times n$ symmetric and positive semidefinite matrices, $\N$ is a nonempty simple closed convex set (e.g., $\N =\{X\in\cS^n\;|\; L\leq X\leq U\}$ with $L,U\in \cS^n$ being given matrices),  $C\in \cS^n$,  $b_E\in \Re^{m_E}$ and $b_I\in \Re^{m_I}$ are given data.

The dual of problem \eqref{dnnsdp} then takes the following form
 \begin{eqnarray}
  \begin{array}{rllll}
    \max\limits_{S,Z,y_E,y_I} &  \ds -\delta_{\N}^*(-Z)  + \langle{b_E},{y_E}\rangle + \langle{b_I},{y_I}\rangle   \\[2mm]
   \mbox{s.t.} & Z + S + \A_E^* y_E + \A_I^*y_I = C, \ S \in \cS^n_+, \quad y_I\geq 0,
   \end{array}
   \label{eq-d-qsdp}
\end{eqnarray}
which is an instance of multi-block composite conic optimization problems.
In order to handle the inequality constraints in \eqref{eq-d-qsdp} so that
it can be solved by the GADMM algorithm together with the sGS technique, we introduce a slack variable $v$ and a positive definite a nonsinglular linear operator $\D^*:\Y\to\Y$ to get the following equivalent optimization problem:
\begin{eqnarray}
\begin{array}{rllll}
\min\limits_{S,Z,y_E,y_I,v} &  \ds \big(\delta_{\N}^*(-Z) +\delta_{\Re_{+}^{m_I}}(v)\big)+ \delta_{\cS^n_{+}}(S) - \langle{b_E},{y_E}\rangle - \langle{b_I},{y_I}\rangle   \\[2mm]
 \mbox{s.t.} & Z  + S + \A_E^* y_E + \A_I^*y_I = C, \  \D^*(v -  y_I) = 0.
   \end{array}
   \label{eq-qsdp-dual}
\end{eqnarray}
Therefore, the DNN-SDP \eqref{dnnsdp} can be solved by the GADMM algorithm via its dual form \eqref{eq-qsdp-dual}, together with the sGS techniques that were illustrated in Algorithm \ref{alg2}.

We use $X\in\cS^{n}$ and $u\in\Re^{m_{I}}$ to denote the dual variables corresponding to the two groups of equality constraints in \eqref{eq-qsdp-dual}, respectively.
Moreover, we denote  the normal cone of $\N$ at $X$ by $N_{\N}(X)$.
Then, the KKT system of problem \eqref{dnnsdp} is given by
\begin{equation}\label{eqopM}
\left\{
\begin{array}{l}
\A_{E}^{*}y_{E}+\A_{I}^{*}y_{I}+S+Z-C=0, \  \A_{E}X-b_{E}=0,
\\[1mm]
0\in N_{\N}(X)+Z,\
X\in\cS^{n}_{+},\
S\in\cS_{+}^{n},\
\langle X,S\rangle=0,
\\[1mm]
\A_{I}X-b_{I}\ge0,\ y_{I}\ge 0,\  \langle \A_{I}x-b,y_{I}\rangle =0.
\end{array}
\right.
\end{equation}
Based on the optimality condition \eqref{eqopM},
we measure the accuracy of a computed candidate solution $(X,Z,S,y_E,y_{I})$ for \eqref{dnnsdp} and its dual \eqref{eq-d-qsdp} via
$$
\begin{array}{c}
\eta_{SDP} =
\max\{\eta_D,\eta_{X},\eta_Z,  \eta_P, \eta_{S},
\eta_{I}\},
\end{array}
$$ where
\begin{equation}
\label{stop:dnnsdp}
\begin{array}{l}
\eta_{D}=\frac{\|\A_{E}^{*}y_{E}+\A_{I}^{*}y_{I}+S+Z-C\|}{1+\|C\|},\
\eta_{X}=\frac{\|X-\Pi_{\N}(X)\|}{1+\|X\|},\
\eta_{Z}=\frac{\|X-\Pi_{\N}(X-Z)\|}{1+\|X\|+\|Z\|},
\\[3mm]
\eta_{P}=\frac{\|\A_{E}X-b_{E}\|}{1+\|b_{E}\|},\
\eta_{S}=\max\big\{\frac{\|X-\Pi_{S_{+}^{n}}(X)\|}{1+\|X\|},
\frac{|\langle X,S\rangle|}{1+\|X\|+\|S\|}\big\},
\\[3mm]
\eta_{I}=\max\big\{
\frac{\|\min(0,y_{I})\|}{1+\|y_{I}\|},
\frac{\|\min(0,\A_{I}X-b_{I})\|}{1+\|b_{I}\|},
\frac{|\langle \A_{I}X-b_{I},y_{I}\rangle|}{1+\|\A_{I}x-b_{I}\|+\|y_{I}\|}\big\}.
\end{array}
\end{equation}
%In addition, we also measure the objective value and the duality gap:
%\begin{equation}
%\label{gap}
%\begin{array}{rl}
%{\rm Obj_{primal}}&:=\frac{1}{2}\langle X,\Q X\rangle+\langle C,X\rangle,\\[1mm]
%%{\rm Obj_{dual}}&:=-\delta_{\N}^*(-Z)-\frac{1}{2}+ \langle{b_E},{y_E}\rangle + %\langle{b_I},{y_I}\rangle,
%\\[1mm]
%\eta_{\rm gap}&:=\frac{\rm Obj_{primal}-Obj_{dual}}{1+|\rm Obj_{primal}|+|Obj_{dual}|}.
%\end{array}
%\end{equation}
Next, we would like to report our numerical results.

\subsection{Numerical Results on DNN-SDP with Many Inequality Constraints}
In our numerical experiments, we first test DNN-SDP problems from the relaxation of a binary integer quadratic (BIQ) programming, with a large number of valid inequalities introduced by Sun et. al. \cite{YANGSIAM} to get tighter bounds, i.e.,
\[
     \begin{array}{l}
    \min \    \frac{1}{2} \langle{Q},{\overline{X}}\rangle + \langle{c},{x}\rangle  \\[5pt]
   \mbox{s.t.}\
        {\rm diag}(\overline{X}) - x = 0, \quad     X = \left(
               \begin{array}{cc}
                 \overline{X} & x \\
                 x^T & 1\\
               \end{array}
             \right)
       \in \cS^n_+,\quad  X\ge 0,\\[3mm]
    \ \quad   x_i  -\overline{X}_{ij} \geq 0,\;  x_j -\overline{X}_{ij} \geq 0,\;
 \overline{X}_{ij} - x_i -x_j \geq -1\;  \forall i < j,  j=2,\dots,n-1. \end{array}
\]
In the above problem,  ${\rm diag}(\overline{X}) $ denotes the vector consisted of the diagonal elements of the $(n-1)\times(n-1)$ matrix $\bar X$. In our test, the  data for {$Q$ and $c$ are taken from the} Biq Mac Library\footnote{\url{http://biqmac.uni-klu.ac.at/biqmaclib.html}} maintained by Wiegele.
The linear operator $\D$ in \eqref{eq-qsdp-dual} is set by $\D:=\alpha\I$.
By writing down the augmented Lagrangian function of  problem \eqref{eq-qsdp-dual}, one can observe that the quadratic (penalty) terms in it takes the from of
$$
\frac{\sigma}{2}
\| Z  + S + \A_E^* y_E + {\A_I^*}y_I -C\|^2
+\frac{\sigma{\alpha^2}}{2}\|v -  y_I\|^2.
$$
Therefore, we can set  choose $\alpha=\sqrt{\|\A_k\|}/2$ so that, numerically, the second term above is relatively stable with respect to the previous one. Of course, one can choose other numbers.
By properly choosing the proximal terms via the techniques provided in \cite[Sect. 7.1]{CHENL}, the corresponding subproblems will have closed-form solutions.
We mention that all these algorithms are tested by running {\sc Matlab} on a HP Elitedesk with one Intel Core i7-4770S Processor
(4 Cores, 8 Threads, 8M Cache, 3.1 to 3.9 GHz) and 8 GB (DDR3-1600 MHz) RAM.

We compare the performance of GADMM with the relaxation parameter $\rho$ ranging from $1$ to $1.9$. The reason is that, on the one hand, during a preliminary test on a few number of problems, we found that the case $\rho<1$ is not interesting in the sense that it is almost no better than $\rho=1$, and, on the other hand,  the design of the algorithm in this paper is intended to use a relatively large $\rho$ to get better numerical performance.
Besides,  the performance of the sGS based semi-proximal ADMM from \cite{YANGSIAM,lisgs} with step length of $1$ and $1.618$\footnote{In fact, we have applied the technique discussed in \cite{YANGSIAM,admmnote} for using a step-length which can sometimes larger than $1.618$ to further improve the computational efficiency.
However, we still call the step-length is $1.618$, just for convenience.} is also reported to serve as the benchmark for comparison .

\begin{figure}
  \includegraphics[width=0.44\textwidth]{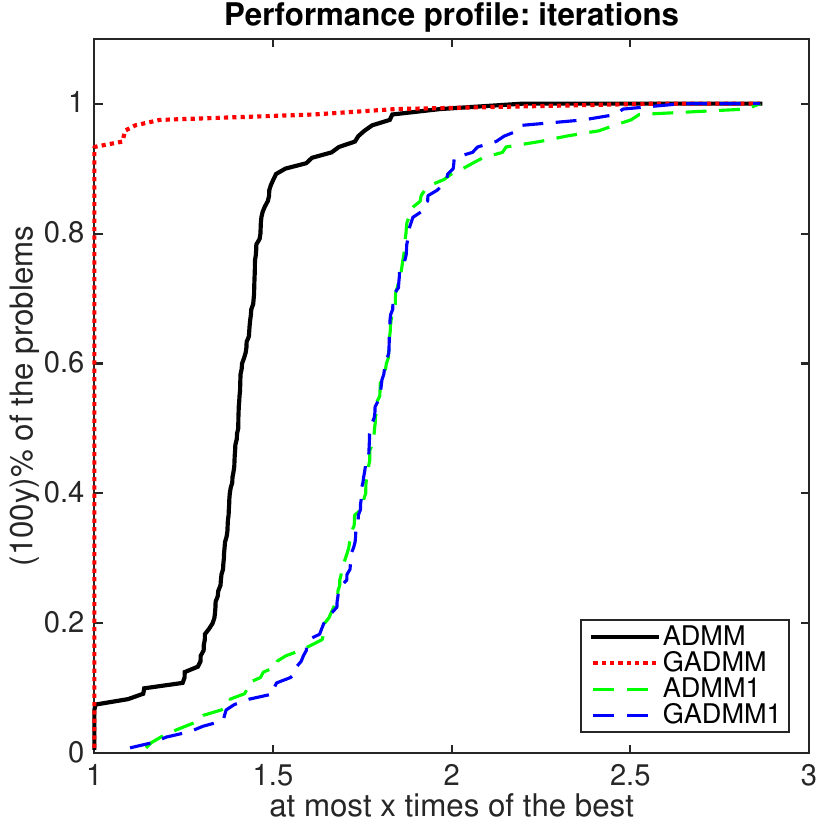}\hspace{0.5cm}
  \includegraphics[width=0.44\textwidth]{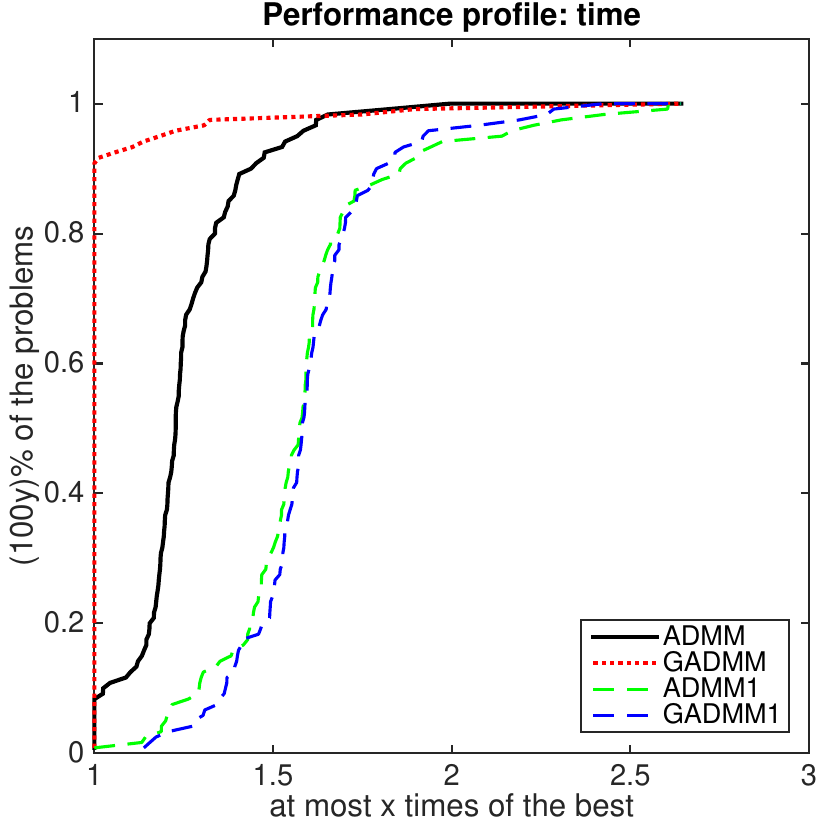}
\caption{The performance profile of  ADMM, GADMM, ADMM1 and  GADMM1 on solving DNN-SDP problems with many inequality constraints.}
\label{fig:1}
\end{figure}

Table \ref{table:sdp} presents the computational results of all the $12$ algorithms, i.e.  GADMM with $\rho=1$, $\rho=1.1$, $\ldots$, $\rho=1.9 $ and sPADMM with $\tau=1$ and $\tau=1.618$. All the algorithms are terminated if $\eta_{SDP}<10^{-6}$, or the maximum iteration number $500,000$ is achieved. We have not presented $\eta_{SDP}$ for each computed instance in this table since almost all the problems are solved to the required accuracy by each of the above mentioned algorithms.
We should mention that the GADMM with $\rho=1$ (denoted by GADMM1) and the sPADMM with $\tau=1$  (denoted by sPADMM1) are essentially the same algorithm, even if the implementations of them are slightly different. As one can observe from Table \ref{table:sdp}, the total computational time of GADMM1 and sPADMM1 are also similar, while  GADMM with $\rho=1.8$ has the minimum computational time.
One can also see that for a certain portion of problems, GADMM with $\rho=1.9$ performs the best, but there are some problems for which this relaxation parameter is too large to provide a satisfactory solution within the prescribed maximum iteration number. Therefore, generally speaking, the merit of large relaxation parameter $\rho$ is apparent, but one should avoid of the risk that brought by too large $\rho$. Therefore, we would like to suggest that $\rho=1.8$ is a desirable choice.

Figure \ref{fig:1} presents the performance profile of four algorithms (GADMM with $\rho=1$ and $\rho=1.8$, denoted by GADMM1 and GADMM, respectively, and sPADMM with $\tau=1$ and $\tau=1.618$, denoted by sPADMM1 and sPADMM, respectively).
From Figure \ref{fig:1} one can see that GADMM1 and sPADMM1 have very similar performance when solving these problems and GADMM performs the best. As a consequence, in the subsequent section, we would like to  test the performance of GADMM  with $\rho=1.8$ only for DNN-SDP problems without linear inequality constraints, and compare it with other algorithms.
Therefore, it is obvious that both the relaxation step of GADMM and large step-length of sPADMM can enhance the numerical performance of sPADMM with unit step-length and the GADMM has the probability to perform even better.

\subsection{Numerical Results on DNN-SDP without Inequality Constraints}
\begin{figure}
  \includegraphics[width=0.46\textwidth]{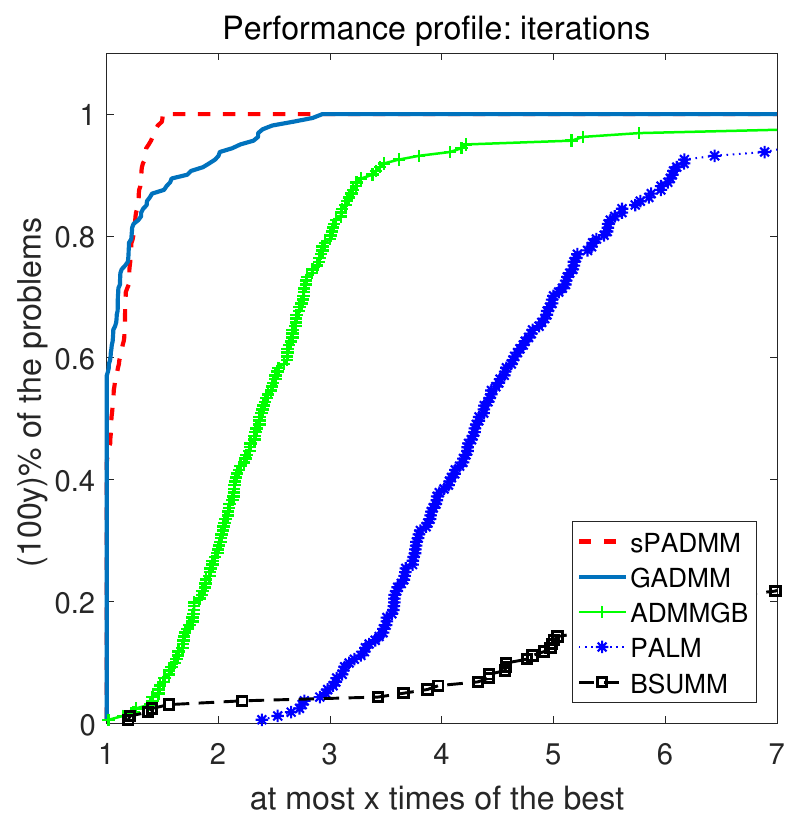}\hspace{0.2cm}
  \includegraphics[width=0.46\textwidth]{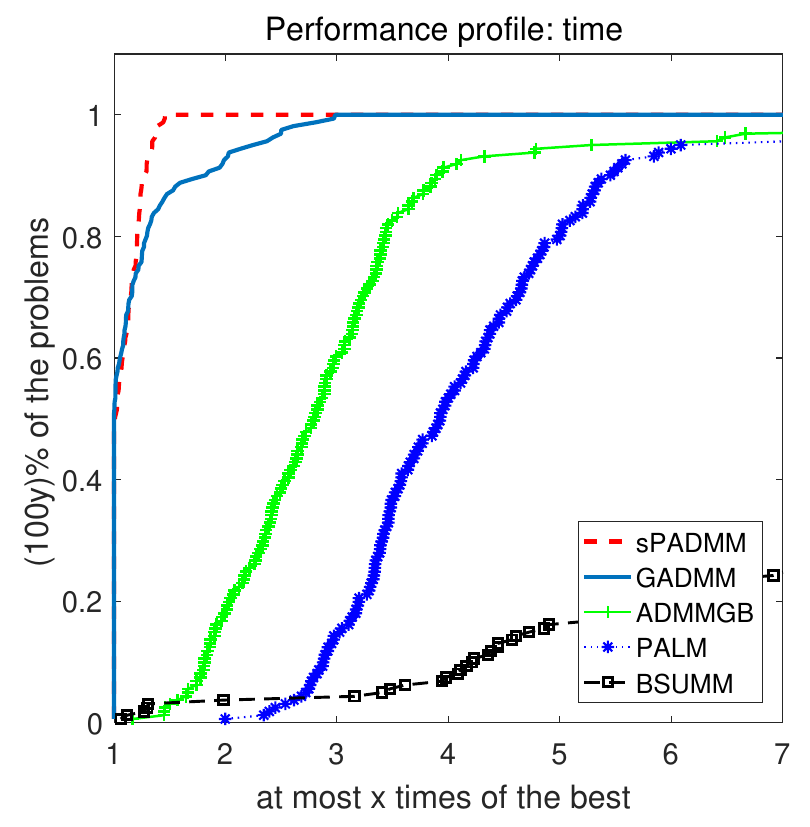}
\caption{The performance profile of sPADMM, GADMM, ADMMGB, PALM and BSUMM on solving DNN-SDP problems without inequality constraints.}
\label{fig:2}
\end{figure}

Next, we report our numerical results on solving DNN-SDP problems without inequality constraints see the performance of the proposed generalized ADMM. In this part, we fix $\rho$ for GADMM as $\rho=1.8$ since the  numerical results reported in the previous subsection suggests that this relaxation parameter performs pretty good.
The detailed description of the $281$ instances of DNN-SDP problems can be found in Sun et al. \cite[Section 5.1.1]{YANGSIAM}.

For comparison with other state-of-the-art algorithms, we also tested the sGS based sPADMM \cite{YANGSIAM,lisgs} with $\tau=1.618$, the ADMM with Gaussian back substitution \cite{he} by He et al. (denoted by ADMMGB), the parallel multi-block ADMM by Deng et al.\cite{dengwei} (denoted by PALM), and the block successive upper bound minimization method of multipliers \cite{hong} by Hong et al. (denoted by BSUMM).
We mention that all the solvers are tested by running {\sc Matlab} on an Apple Macbook Air Laptop with one Intel Core i7-5650U Processor
(2 Cores, 4 Threads, 4M Cache, 2.2 to 3.1 GHz) and 8 GB (DDR3-1600 MHz) RAM.

For ADMMGB, we set the back substitution parameter as 0.99.
For PALM, we used the proximal-linear type (c.f. \cite[Theorem 2.1]{dengwei}) proximal terms while the adaptive parameter tuning strategy in \cite[Sec. 3]{dengwei} was not used. The reason is that, during a preliminary test, we found that the improvement is very limited but sometimes it can dramatically increase the iteration numbers.
For BSUMM, it is very hard to provide a uniform way to choose the step-length. Therefore, although we have tried our best to find a strategy of providing the step-length, the computational results are not promising.
All the tested algorithms are terminated if $\eta_{\textup{SDP}}<10^{-6}$,  where
$$\eta_{\textup{SDP}}:=
\max\{\eta_D,\eta_{X},\eta_Z,  \eta_P, \eta_{S}\},$$
with $\eta_{X},\eta_Z,  \eta_P, \eta_{S}$ being defined in \eqref{stop:dnnsdp}
and $$\eta_{D}:=\frac{\|\A_{E}^{*}y_{E}+S+Z-C\|}{1+\|C\|},$$
or the maximum iteration number $500,000$ is achieved.

Table \ref{table:sdpeq} presents the computational results of
sPADMM, GADMM, ADMMGB, PALM and BSUMM. The corresponding performance profiles are presented in Figure \ref{fig:2}.
As one can observe that, the sPADMM performs the best and the GADMM proposed in this paper has similar performances for about $80\%$ of the tested problems. However, both sPADMM and GADMM are far superior than ADMMGB, PALM and BSUMM.
The reason behind the performance profile can be described as follows. Both sPADMM and GADMM used certain over-relaxation techniques and the sGS techniques for generating relatively small semi-proximal terms. The ADMMGB does not rely on proximal terms, but it can take advantage of some over-relaxation techniques. For PALM, the proximal term is much larger than other algorithms so that its efficiency is even worse than ADMMGB. Finally, the requirements for sufficiently small step-length plays the key role of making BSUMM inefficient, especially for solving problems that can not be handled by first-order methods within a small number of iterations.

%%%%%%%%%%%%%%%%%%%%%%%%%%%%%%%%%%%%%%%%%%%%%%%%%%%%%%%%%%%5
\section{Conclusions}\label{result}
%%%%%%%%%%%%%%%%%%%%%%%%%%%%%%%%%%%%%%%%%%%%%%%%%%%%%%%%%%%%
In a nutshell, this paper provides the convergence analysis for a variant of generalized ADMM with semi-proximal terms and applies it to a class of multi-block convex composite conic optimization problems. The method studied in this paper relaxes both the primal and the dual variables simultaneously and, more importantly, has the advantage of solving multi-block problems by using the techniques discussed in Sect. \ref{appl}. The semi-proximal terms for the generalized ADMM considered in this paper is more natural than the other variant since the most recent values of variables are always used in the semi-proximal terms. Numerically, we conducted a comparison between the proposed algorithm and four state-of-the-art first order solvers.
The computational results suggest that the proposed method performs similar to the sPADMM and superior than the rest.
Therefore, this paper not only provided an efficient algorithm, but also can help to get clear the performance of several first-order methods, at least for DNN-SDP problems.
We should mention that although there are many solvers \cite{pov,ZHAOSIAM,MMPC,MCOAP,WENMPC,YANGSIAM} for solving the standard linear SDP problems,  we only choose the block sGS based semi-Proximal ADMM for comparison as it has been proven in \cite{YANGSIAM} to be the most efficient ADMM-type method for solving DNN-SDP considered here.
Moreover, solving subproblems inexactly can also improve  the performance of ADMM-type algorithms, as was observed in \cite{CHENL}.
Therefore, it is of interest to combine the relaxation scheme proposed in this paper and the techniques of using approximate solutions of subproblems, and we leave it as our future work.

\section*{Acknowledgements}
We would like to thank the anonymous referees and the associate editor for their useful comments and suggestions
which improved this paper greatly. We are very grateful to Professor Defeng Sun at the National University of Singapore for sharing his knowledge with us on topics covered in this paper and beyond. The research of Y. Xiao and L. Chen was supported by the China Scholarship Council while they were visiting the National University of Singapore. The research of Y. Xiao was supported by the Major State Basic Research Development Program of China (973 Program) (Grant No. 2015CB856003), and the National Natural Science Foundation of China (Grant No. 11471101). The research of L. Chen was supported by the Fundamental Research Funds for Central Universities and the National Natural Science Foundation of China (Grant No. 11271117).
The research of D. Li was supported by the National Natural Science Foundation of China (Grant No. 11371154).

%References
% BibTeX users please use
%\bibliographystyle{spmpsci_unsrt}      % mathematics and physical sciences

\begin{thebibliography}{}

\bibitem{CHENPHD} Chen, C.H.: Numerical algorithms for a class of matrix
norm approximation problems. Ph.D. Thesis, Department of Mathematics, Nanjing University, Nanjing, China. \url{http://www.math.nus.edu.sg/~matsundf/Thesis_Caihua.pdf} (2012)

\bibitem{CHENL} Chen, L., Sun, D.F. and Toh, K.-C.: An efficient inexact symmetric Gauss-Seidel based majorized
ADMM for high-dimensional convex composite conic programming. Math. Program. 161(1), 237--270 (2017)% doi: 10.1007/s10107-016-1007-5, (2016)

\bibitem{admmnote} Chen, L., Sun, D.F. and Toh, K.-C.: A note on the convergence of ADMM for linearly constrained convex optimization Problems. Comput. Optim. Appl. 66(2), 327--343 (2017)% doi:10.1007/s10589-016-9864-7, (2016)

\bibitem{CUIST} Cui, Y., Li, X.D.,  Sun, D.F. and Toh, K.-C.: On the convergence properties of a majorized ADMM for linearly constrained convex optimization problems with coupled objective functions. J. Optim. Theory Appl. 169(3), 1013--1041 (2016) %arXiv:1502.00098, (2015)

\bibitem{dengwei}
Deng, W., Lai, M.-J., Peng, Z. and Yin, W.:
Parallel multi-block ADMM with $o(1/k)$ convergence.
\newblock J. Sci. Comput. DOI: 10.1007/s10915-016-0318-2 (2016)


\bibitem{ECKSTEINOMS} Eckstein, J.:
\newblock Some saddle-function splitting methods for convex
programming.
\newblock Optim. Methods Softw. 4, 75--83 (1994)

\bibitem{ECKSTEIN92} Eckstein, J. and Bertsekas, D.P.:
\newblock On the Douglas-Rachford splitting method and the
proximal point algorithm for maximal monotone operators.
\newblock Math. Program. 55, 293--318 (1992)

\bibitem{eck12}
Eckstein, J. and Yao, W.:
\newblock Understanding the convergence of the alternating direction method of multipliers: Theoretical and computational perspectives.
\newblock Pac. J. Optim. 11(4), 619--644 (2014)

\bibitem{SEMP13}
Fazel, M., Pong, T.K, Sun, D.F. and Tseng, P.:
\newblock Hankel matrix rank minimization with applications in system identification and realization.
\newblock SIAM J. Matrix Anal. Appl. 34, 946-977 (2013)

\bibitem{FORTIN83}
Fortin, M., Glowinski, R.:
Augmented Lagrangian Methods: Applications to the Numerical Solution of Boundary-Value Problems. Studies in Mathematics and its Applications, vol. 15. Translated from
French by Hunt, B. and Spicer, D.C. Elsevier Science Publishers B.V. (1983)




\bibitem{GABAY83} Gabay, D.: Applications of the method of multipliers to variational inequalities in Augmented
Lagrangian Methods: Applications to the Numerical Solution of Boundary-Value Problems.
\newblock In: Fortin, M. and Glowinski, R. (eds.): Studies in Mathematics and Its Applications, vol. 15, pp. 299--331.
Elsevier (1983)

\bibitem{GABAY76} Gabay, D. and Mercier, B.:
\newblock A dual algorithm for the solution of nonlinear variational
problems via finite element approximation.
\newblock Comput. Math. Appl. 2, 17-40 (1976)

\bibitem{GLOWINSKI80}
Glowinski, R.:
\newblock Lectures on numerical methods for non-linear variational problems.
\newblock Published for the Tata Institute of Fundamental Research, Bombay [by] Springer-Verlag (1980).

\bibitem{globook14}
Glowinski, R.:
\newblock On alternating direction methods of multipliers: A historical
  perspective.
\newblock In Fitzgibbon, W., Kuznetsov, Y.A., Neittaanmaki, P. and Pironneau, O. (eds.):  Modeling, Simulation and Optimization for Science and Technology, pp. 59--82. Springer, Netherlands (2014)

\bibitem{GLOWINSKI75}
Glowinski, R. and Marroco, A.:
\newblock Sur l'approximation, par \'el\'ements finis d'ordre un, et la r\'esolution, par p\'enalisation-dualit\'e d'une classe de probl\`emes de Dirichlet non lin\'eaires.
\newblock Revue fran\c{c}aise d'atomatique, Informatique Recherche Op\'erationelle. Analyse Num\'erique, 9(2), 41--76 (1975)

\bibitem{he}
He, B.,Tao, M. and Yuan, X.:
\newblock Alternating direction method with Gaussian back substitution for separable convex programming.
\newblock SIAM J. Optim. 22(2), 313--340  (2012)


\bibitem{hestenes69}
Hestenes, M.:
 Multiplier and gradient methods.
J. Optim. Theory Appl. 4(5), 303--320 (1969)

\bibitem{hong}
Hong, M., Chang, T.-H., Wang, X., Razaviyayn, M., Ma, S. and Luo, Z.-Q.:
A block successive upper bound minimization method of multipliers for linearly constrained convex optimization.
arXiv:1401.7079 (2014)




\bibitem{limajorize}
Li, M., Sun, D.F. and Toh, K.-C.:
\newblock A majorized ADMM with indefinite proximal terms for linearly constrained convex composite optimization. SIAM J. Optim. 26, 922--950 (2016)


\bibitem{MASIAMOPT}
Lin, M.,  Ma S.Q. and Zhang S.Z.:
\newblock On the global linear convergence of the ADMM with multi-block variables. SIAM J. Optim. 25(3): 1478--1497 (2015)

\bibitem{MACOAP}
Lin, M.,  Ma S.Q. and Zhang S.Z.:
\newblock Iteration complexity analysis of multi-block ADMM for a family of convex minimization without strong convexity.
\newblock Journal Sci. Comput. 69, 52--81, (2016)


\bibitem{XDLIMP} Li, X.D., Sun, D.F. and Toh, K-C.: A Schur complement based semi-proximal ADMM for convex quadratic conic programming and extensions. Math. Program. 155,  333-373 (2016)

\bibitem{lithesis}
Li, X.D.: A two-phase augmented Lagrangian method for convex composite quadratic programming, PhD thesis, Department of Mathematics, National University of Singapore (2015)


\bibitem{MMPC} Monteiro, R.D.C., Ortiz, C. and Svaiter, B.F.:
\newblock A first-order block-decomposition method for solving two-easy-block structured semidefinite programs.
\newblock Math. Program. Comput. 6, 103-150 (2014).

\bibitem{MCOAP} Monteiro, R.D.C., Ortiz, C. and Svaiter, B.F.:
\newblock Implementation of a block-decomposition algorithm for solving large-scale conic semidefinite programming problems.
\newblock Comput. Optim. Appl. 57, 45-69 (2014)






\bibitem{lisgs}
Li, X.D., Sun, D.F. and Toh. K.-C.:
QSDPNAL: A two-phase Newton-CG proximal augmented Lagrangian
method for convex quadratic semidefinite programming problems, arXiv:1512.08872, (2015)

\bibitem{pov}
Povh, J., Rendl F. and Wiegele, A.:
A Boundary Point Method to Solve Semidefinite Programs.
Comput. 78, 277–286 (2006)


\bibitem{powell}
Powell, M.J.D.:
\newblock A method for nonlinear constraints in minimization problems.
\newblock In: Fletcher, R. (eds.):, Optimization, pp. 283--298. Academic Press (1969)




\bibitem{rocbook}
Rockafellar, R.T.:
\newblock Convex Analysis.
\newblock Princeton University Press (1970)

\bibitem{rockafellar}
Rockafellar, R.T.:
\newblock Augmented Lagrangians and applications of the proximal point algorithm in convex programming.
\newblock {Math. Oper. Res.}, 1(2), 97--116 (1976)

\bibitem{roc76}
Rockafellar, R.T.:
\newblock Monotone operators and the proximal point algorithm.
\newblock SIAM J. Control Optim. 14(5), 877--898 (1976)

\bibitem{ROC78} Rockafellar, R.T.: Monotone operators and augmented lagrangian methods in nonlinear programming. in Nonlinear Programming 3, O.L. Mangasarian, R.M. Meyer and S.M. Robinson, Editors,  Academic Press, New York, (1977), 1-25.

\bibitem{YANGSIAM} Sun, D.F., Toh, K.-C. and Yang, L.:
\newblock A convergent 3-block semi-proximal alternating direction method of multipliers for conic programming with 4-type constraints.
\newblock SIAM J. Optim. 25, 882--915 (2015)


\bibitem{WENMPC}
Wen, Z., Goldfarb, D. and Yin, W.:
\newblock Alternating direction augmented Lagrangian methods for semidefinite programming.
\newblock Math. Program. Comput. 2, 203--230 (2010)



\bibitem{ZHAOSIAM}
Zhao, X.Y., Sun, D.F. and Toh, K.-C.:
\newblock A Newton-CG augmented Lagrangian method for semidefinite programming.
\newblock SIAM J. Optim. 20, 1737--1765 (2010)



\end{thebibliography}

\textheight=24.5cm
\tiny
\newpage
\begin{landscape}
\begin{center}
\begin{longtable}{lllr}
\caption{Numerical results of ADMM and GADMM on DNN-SDP with equality and inequality constraints (accuracy $= 10^{-6}$)}
\label{table:sdp}
\vspace{0.5em}\\
\hline
\mc{2}{c}{}
&\mc{1}{c}{GADMM iteration $||$ sPADMM iteration}
&\mc{1}{c}{GADMM time $||$ sPADMM time}
\\
%%%%%%%%%%%%%%%%%%%%%%%%%%%%% 3rd line %%%%%%%%%%%%%%%%%%%%%%%%%%%%%%%%%%
\mc{1}{l}{problem}
&\mc{1}{c}{$m_E;m_{I};n_s$}
&\mc{1}{c}{$1|1.1|1.2|1.3|1.4|1.5|1.6|1.7|1.8|1.9||1|1.618$}
&\mc{1}{c}{$1|1.1|1.2|1.3|1.4|1.5|1.6|1.7|1.8|1.9||1|1.618$}
\\
\midrule[1pt]
\endfirsthead
\multicolumn{3}{c}{Table~\thetable~（continued）}\vspace{0.5em}\\
\toprule[1.5pt]
\mc{2}{c}{}
&\mc{1}{c}{GADMM iteration $||$ sPADMM iteration}
&\mc{1}{c}{GADMM time $||$ sPADMM time}
\\
%%%%%%%%%%%%%%%%%%%%%%%%%%%%% 3rd line %%%%%%%%%%%%%%%%%%%%%%%%%%%%%%%%%%
\mc{1}{l}{problem}
&\mc{1}{c}{$m_E;m_{I};n_s$}
&\mc{1}{c}{$1|1.1|1.2|1.3|1.4|1.5|1.6|1.7|1.8|1.9||1|1.618$}
&\mc{1}{c}{$1|1.1|1.2|1.3|1.4|1.5|1.6|1.7|1.8|1.9||1|1.618$}\\
\midrule[1pt]
\endhead
\bottomrule[1.5pt]
\endfoot
%%%%%%%%%%%%%%%%%%%%%%%%%%%%%%%%%%%%%%%%%%%%%%%%%%%%%%%%%%%%%%%
be100.1 &101;14850;101 &32260$|$27492$|$26087$|$23236$|$20134$|$22263$|$20574$|$17768$|$19066$|$18256$||$28090$|$23776&2:04$|$1:53$|$1:47$|$1:36$|$1:22$|$1:35$|$1:32$|$1:19$|$1:23$|$1:16$|$$|$1:40$|$1:25
\\[2pt]
be100.2 &101;14850;101 &13834$|$20683$|$18986$|$17790$|$16674$|$14156$|$14447$|$8854$|$12896$|$8242$||$13594$|$11878&48.4$|$1:23$|$1:17$|$1:12$|$1:08$|$57.3$|$1:02$|$35.7$|$54.6$|$32.3$|$$|$46.7$|$41.3
\\[2pt]
be100.3 &101;14850;101 &16672$|$20174$|$14361$|$13636$|$13240$|$12589$|$11249$|$10520$|$9419$|$9235$||$16678$|$13619&1:01$|$1:22$|$56.3$|$53.4$|$52.3$|$50.5$|$46.6$|$43.4$|$38.5$|$36.7$|$$|$59.3$|$47.9
\\[2pt]
be100.4 &101;14850;101 &31065$|$28275$|$26317$|$23938$|$21912$|$20766$|$19654$|$18313$|$17075$|$16490$||$31336$|$24460&1:57$|$1:55$|$1:48$|$1:38$|$1:30$|$1:25$|$1:25$|$1:19$|$1:13$|$1:10$|$$|$1:56$|$1:30
\\[2pt]
be100.5 &101;14850;101 &14980$|$14006$|$12729$|$12073$|$11101$|$10367$|$9615$|$9085$|$8653$|$8044$||$15158$|$11790&55.7$|$56.3$|$51.2$|$48.7$|$44.5$|$41.7$|$40.5$|$37.9$|$36$|$32.6$|$$|$54.5$|$42.6
\\[2pt]
be100.6 &101;14850;101 &17765$|$15081$|$14523$|$12858$|$12640$|$11504$|$11094$|$10536$|$9417$|$8622$||$18029$|$13553&1:04$|$59$|$56.4$|$50.4$|$49.7$|$44.8$|$45.5$|$43.3$|$38.5$|$34.2$|$$|$1:04$|$47.8
\\[2pt]
be100.7 &101;14850;101 &25915$|$18166$|$19663$|$14272$|$18011$|$12568$|$12048$|$10917$|$10452$|$12373$||$25221$|$19095&1:34$|$1:10$|$1:17$|$54.4$|$1:11$|$47.9$|$49.1$|$43.9$|$41.7$|$52.6$|$$|$1:29$|$1:08
\\[2pt]
be100.8 &101;14850;101 &18593$|$17086$|$15784$|$14664$|$13674$|$13052$|$11716$|$11021$|$10655$|$9608$||$18613$|$15048&1:07$|$1:05$|$1:02$|$56.3$|$57.5$|$55.5$|$47.8$|$44.7$|$42.8$|$37.9$|$$|$1:05$|$52.7
\\
\hline \hline

\end{longtable}

\tiny
\begin{longtable}{lllccccr}
\caption{Numerical results of sPADMM GADMM BSUMM ADMMGB and PALM on DNN-SDP with equality  constraints (accuracy $= 10^{-6}$)}
\label{table:sdpeq}
\vspace{0.5em}\\
\hline
\mc{3}{c}{}
&\mc{1}{c}{iteration}
&\mc{1}{c}{$\eta_{SDP}$}
&\mc{1}{c}{time}
\\
%%%%%%%%%%%%%%%%%%%%%%%%%%%%% 3rd line %%%%%%%%%%%%%%%%%%%%%%%%%%%%%%%%%%
\mc{1}{l}{problem}
&\mc{1}{l}{$m_E$} %(1) problem information ( name, size )
&\mc{1}{c}{$n_s$}
&\mc{1}{c}{$P|G|B|Gb|Pa$}
&\mc{1}{c}{$P|G|B|Gb|Pa$}
&\mc{1}{c}{$P|G|B|Gb|Pa$}

\\
\midrule[1pt]
\endfirsthead
\multicolumn{3}{c}{Table~\thetable（continued）}\vspace{0.5em}\\
\toprule[1.5pt]
\mc{3}{c}{}
&\mc{1}{c}{iteration}
&\mc{1}{c}{$\eta_{SDP}$}
&\mc{1}{c}{time}
\\
%%%%%%%%%%%%%%%%%%%%%%%%%%%%% 3rd line %%%%%%%%%%%%%%%%%%%%%%%%%%%%%%%%%%
\mc{1}{l}{problem}
&\mc{1}{c}{$m_E;m_{I}$} %(1) problem information ( name, size )
&\mc{1}{c}{$n_s$}
&\mc{1}{c}{$P|G|B|Gb|Pa$}
&\mc{1}{c}{$P|G|B|Gb|Pa$}\\
\midrule[1pt]
\endhead
\bottomrule[1.5pt]
\endfoot
%%%%%%%%%%%%%%%%%%%%%%%%%%%%%%%%%%%%%%%%%%%%%%%%%%%%%%%%%%%%%%%
theta4 &1949&200 &324$|$323$|$519$|$1155$|$2085&9.589398e-07$|$9.634160e-07$|$9.802864e-07$|$9.437244e-07$|$9.954982e-07
&1.9$|$1.9$|$10.4$|$3.6$|$5.6
\\[2pt]  
theta42 &5986&200 &184$|$140$|$269$|$613$|$978&9.684599e-07$|$9.762212e-07$|$9.894012e-07$|$9.960292e-07$|$9.926681e-07
&1.1$|$1$|$4.9$|$1.8$|$3
\\[2pt]  
theta6 &4375&300 &278$|$341$|$553$|$1386$|$1746&9.232747e-07$|$9.881228e-07$|$9.988977e-07$|$9.890359e-07$|$9.997234e-07
&3.5$|$4.5$|$20.7$|$8.4$|$16.3
\\[2pt]  
theta62 &13390&300 &166$|$160$|$222$|$606$|$898&9.274991e-07$|$8.726820e-07$|$9.648441e-07$|$8.929543e-07$|$9.924620e-07
&2.1$|$2.3$|$11.8$|$3.5$|$7.5
\\[2pt]  
theta8 &7905&400 &352$|$327$|$578$|$1431$|$1626&9.308277e-07$|$9.722639e-07$|$9.742229e-07$|$9.334721e-07$|$9.921152e-07
&10.6$|$10.8$|$47.3$|$18.9$|$38
\\[2pt]  
theta82 &23872&400 &156$|$172$|$228$|$554$|$821&9.527985e-07$|$8.787174e-07$|$9.341241e-07$|$9.290212e-07$|$9.927168e-07
&4.4$|$5.1$|$21.4$|$8$|$14
\\[2pt]  
theta83 &39862&400 &134$|$118$|$203$|$429$|$831&9.211792e-07$|$9.578097e-07$|$9.576721e-07$|$9.986597e-07$|$9.960868e-07
&3.9$|$3.7$|$21.6$|$7.4$|$11.7
\\[2pt]  
theta10 &12470&500 &357$|$392$|$636$|$1586$|$1633&9.905064e-07$|$8.557371e-07$|$9.936271e-07$|$9.671262e-07$|$9.979535e-07
&16$|$18.5$|$1:11$|$34.5$|$1:03
\\
\hline 

\end{longtable}
\end{center}
\end{landscape}

\end{document}